\documentclass{amsart}
\usepackage{amsmath,amsthm,amssymb,hyperref,mathrsfs}
\usepackage{aliascnt}
\usepackage{lmodern}
\usepackage[T1]{fontenc}

\usepackage{xcolor}
\definecolor{dblue}{rgb}{0,0,0.70}
\hypersetup{
	unicode=true,
	colorlinks=true,
	citecolor=dblue,
	linkcolor=dblue,
	anchorcolor=dblue
}

\makeatletter
\expandafter\g@addto@macro\csname th@plain\endcsname{%
		\thm@notefont{\bfseries}
	}%
\expandafter\g@addto@macro\csname th@remark\endcsname{%
		\thm@headfont{\bfseries}
	}%
\makeatother

\newtheorem{theorem}{Theorem}[section]	
\newtheorem*{theorem*}{Theorem}

\newaliascnt{lemma}{theorem}
\newtheorem{lemma}[lemma]{Lemma}
\aliascntresetthe{lemma}

\newaliascnt{proposition}{theorem}
\newtheorem{proposition}[proposition]{Proposition}
\aliascntresetthe{proposition}

\newaliascnt{corollary}{theorem}
\newtheorem{corollary}[corollary]{Corollary}
\aliascntresetthe{corollary}

\theoremstyle{remark}

\newtheorem*{remark}{Remark}

\newaliascnt{definition}{theorem}
\newtheorem{definition}[definition]{Definition}
\aliascntresetthe{definition}

\newaliascnt{example}{theorem}

\aliascntresetthe{example}

\renewcommand{\restriction}{\mathbin\upharpoonright}

\newcommand{\axiom}[1]{\mathsf{#1}} 
\newcommand{\ZFC}{\axiom{ZFC}}
\newcommand{\AC}{\axiom{AC}}
\newcommand{\DC}{\axiom{DC}}
\newcommand{\ZF}{\axiom{ZF}}

\newcommand{\Ord}{\mathrm{Ord}}
\newcommand{\HOD}{\mathrm{HOD}}
\newcommand{\GCH}{\axiom{GCH}}
\newcommand{\CH}{\axiom{CH}}

\newcommand{\IS}{\axiom{IS}}
\newcommand{\HS}{\axiom{HS}}
\newcommand{\KWP}{\axiom{KWP}}
\newcommand{\SVC}{\axiom{SVC}}

\DeclareMathOperator{\cf}{cf}
\DeclareMathOperator{\dom}{dom}
\DeclareMathOperator{\rng}{rng}
\DeclareMathOperator{\supp}{supp}
\DeclareMathOperator{\rank}{rank}
\DeclareMathOperator{\sym}{sym}
\DeclareMathOperator{\fix}{fix}

\DeclareMathOperator{\id}{id}
\DeclareMathOperator{\aut}{Aut}

\DeclareMathOperator{\Add}{Add}
\DeclareMathOperator{\SC}{SC}

\newcommand{\forces}{\mathrel{\Vdash}}
\newcommand{\nforces}{\mathrel{\not{\Vdash}}}
\newcommand{\gaut}[1]{{\textstyle\int_{#1}}}

\newcommand{\PP}{\mathbb{P}}
\newcommand{\power}{\mathcal{P}}

\newcommand{\QQ}{\mathbb{Q}}

\newcommand{\CC}{\mathbb C}
\newcommand{\cF}{\mathcal F}
\newcommand{\sF}{\mathscr F}
\newcommand{\cG}{\mathcal G}
\newcommand{\cB}{\mathcal B}
\newcommand{\sG}{\mathscr G}

\newcommand{\jbd}{\mathcal J_{bd}}
\newcommand{\da}{{\downarrow}}

\newcommand{\tup}[1]{\langle#1\rangle}

\newcommand{\middd}{\mathrel{}\middle|\mathrel{}}

\author{Asaf Karagila}
\thanks{This paper is part of the author's Ph.D.\ written in the Hebrew University of Jerusalem under the supervision of Prof.\ Menachem Magidor.}
\thanks{This research was partially done whilst the author was a visiting fellow at the Isaac Newton Institute for Mathematical Sciences in the programme `Mathematical, Foundational and Computational Aspects of the Higher Infinite' (HIF) funded by EPSRC grant EP/K032208/1.}
\address{Einstein Institute of Mathematics.
Edmond J. Safra Campus, Givat.
The Hebrew University of Jerusalem.
Jerusalem, 91904, Israel}
\curraddr{School of Mathematics, University of East Anglia. Norwich, NR4~7TJ, UK}
\email{karagila@math.huji.ac.il}
\urladdr{http://karagila.org}
\date{June 20, 2018}

\title[The Bristol Model]{The Bristol Model: an abyss called a Cohen real}

\begin{document}
\begin{abstract}
We construct a model $M$ of $\ZF$ which lies between $L$ and $L[c]$ for a Cohen real $c$ and does not have the form $L(x)$ for any set $x$. This is loosely based on the unwritten work done in a Bristol workshop about Woodin's HOD Conjecture in 2011. The construction given here allows for a finer analysis of the needed assumptions on the ground models, thus taking us one step closer to understanding models of $\ZF$, and the HOD Conjecture and its relatives. This model also provides a positive answer to a question of Grigorieff about intermediate models of $\ZF$, and we use it to show the failure of Kinna--Wagner Principles in $\ZF$.
\end{abstract}
\maketitle    
\setcounter{tocdepth}{1}
\tableofcontents
\section{Introduction}
\subsection{The Bristol workshop}
There are many ways to begin this introduction. This is certainly not one of them. In 2011, Philip Welch hosted a workshop in Bristol whose aim was to see whether or not Woodin's HOD Conjecture implies his Axiom of Choice Conjecture. The participants were Andrew Brooke-Taylor, James Cummings, Moti Gitik, Menachem Magidor, Ralf Schindler, Matteo Viale, Philip Welch, and Hugh Woodin (hence forth known as \textit{The Bristol Group}).

They began the workshop by trying to understand how would one refute the hypothesis which they were set to prove. This led them to the construction of a model $M$ which lies between $L$ and $L[c]$, where $c$ is a Cohen real over $L$, such that $M\neq L(x)$ for all $x$. Certainly, this is a strange model, and its construction show that adding a single Cohen real will have incredible consequences for the structure of intermediate models of $\ZF$, whereas the $\ZFC$ models are all defined from a single real, as follows from the intermediate model theorem (Lemma~15.43 in \cite{Jech:ST2003}). It also answers a question of Grigorieff from \cite[p.~471]{Grigorieff:1975}, whether or not can such an intermediate extension exist.\footnote{As remarked by Grigorieff, Solovay proved that from $0^\#$ we can define an $L$-generic filter for an Easton product violating $\GCH$ on a cofinal class of $L$-cardinals. However, $0^\#$ is certainly not generic over $L$.}

The Bristol group, however, did not fully write the details of the construction, and only sketched the basic ideas needed for the construction. Their approach used relative constructibility to define the models. Instead, in the author's Ph.D.\ dissertation, a technique for iterating symmetric extensions was developed, in part to describe the construction of this very model.

Roughly, the original idea can be described as doing a Jensen-coding in reverse from a Cohen real. Namely, when forcing with a Jensen-coding, we code the entire universe to be constructible from a single real, whereas in this case we begin with a Cohen real, and we ``decode'' information from it. The decoding process involves fixing almost disjoint families with particular properties, and using that to define an $\omega_1$-sequence of sets of reals, then immediately forgetting the sequence itself and only recalling parts of it, only to use that to define an $\omega_2$-sequence of sets of sets of reals, and so on. Limit steps are particularly challenging, as there is no obvious way to continue this construction without adding bounded subsets.

The Bristol group, however, noted that one can prove the existence of PCF-theoretic object from square principles, which can then be used to somehow ``condense the construction up to the limit step'', and so proceed with the construction through limit steps.
\subsection{In this paper}
As mentioned earlier, in this paper we take an entirely different approach to construct the model. Instead of relative constructibility and definability-related arguments, we construct the model as an iteration of symmetric extensions, using a framework developed by the author in \cite{Karagila:2016}.

The first part of the paper will cover the basic knowledge needed about symmetric extensions, as well as a detailed construction of the first two steps of the iteration: adding the Cohen real, defining the first intermediate model, and forcing over it to find the $\omega_1$-sequence of sets of reals.

We then move to describe the needed tools for the general construction. This means the relevant definitions generalized from the first symmetric system, the PCF-related objects, and of course an outline of the mechanisms of iterated symmetric extensions needed for this construction.

The paper ends with two general points of interest: the first is the exploration of Kinna--Wagner Principles in the Bristol model, as well as the needed proofs for their complete failure. The second is the observation that we did not use the fact that $V=L$ in the ground model to its full extent. We will then consider what sort of other ground models can be used, what sort of large cardinals can be preserved when moving to the Bristol model from the ground model, and of course how the HOD Conjecture and the Axiom of Choice Conjecture relate to this general construction. We then leave the reader with a few open questions which arise naturally from this construction.

\section{Baby step: taking a hard close look at the first two steps}
\subsection{Quick recap of symmetric extensions}
Let $\PP$ be a notion of forcing (i.e., a preordered set with a maximum element $1_\PP$), recall that $\PP$-names are defined by recursion $V^\PP_\alpha=\bigcup_{\beta<\alpha}\power(\PP\times V^\PP_\beta)$, and we denote $\PP$-names by $\dot x$. Any $x\in V$ has a canonical name defined recursively by $\check x=\{\tup{1_\PP,\check y}\mid y\in x\}$. 

Speaking of canonical names, if we have a set of $\PP$-names in $V$ and we want to turn it into a name, the easiest way to do so is by taking the name for the set of interpretations of these names. We will often want to do just that, and a uniform notation will be useful. If $\{\dot x_i\mid i\in I\}$ is a collection of $\PP$-names, then $\{\dot x_i\mid i\in I\}^\bullet$ is defined to be the $\PP$-name, $\{\tup{1_\PP,\dot x_i}\mid i\in I\}$. We also extend this notation to ordered pairs and sequences, to be interpreted appropriately. Using this notation we can, for example, redefine the canonical name $\check x$ as $\{\check y\mid y\in x\}^\bullet$.

We also set the terminology that a condition $p$ or a $\PP$-name $\dot y$ \textit{appears} in $\dot x$, if there is an ordered pair in $\dot x$ with $p$ in the left coordinate, or $\dot y$ in the right.

Suppose that $\pi$ is an automorphism of $\PP$. We can extend $\pi$ to a permutation $\widetilde{\pi}$ of the $\PP$-names by recursion,\[\widetilde{\pi}\dot x=\{\tup{\pi p,\widetilde{\pi}\dot y}\mid \tup{p,\dot y}\in\dot x\}.\]
From this point we will write $\pi$ instead of $\widetilde{\pi}$, since there is no notation ambiguity between conditions of $\PP$ and $\PP$-names.

\begin{lemma}[The Symmetry Lemma]
Suppose that $\pi$ is an automorphism of $\PP$, $p\in\PP$ and $\dot x$ is a $\PP$-name. Then for every formula in the language of forcing,
\[p\forces_\PP\varphi(\dot x)\iff\pi p\forces_\PP\varphi(\pi\dot x).\qed\]
\end{lemma}

Suppose that $\sG$ is a group we say that $\sF$ is a \textit{filter of subgroups} over $\sG$ if every $H\in\sF$ is a subgroup of $\sG$, and $\sF$ is closed under supergroups, and finite intersection. We will usually require that $\sF$ be a proper filter, namely $\{1_\sG\}\notin\sF$. In other words, we take a filter over $\sG$ and replace each set in the filter by the subgroup it generates. We say that $\sF$ is a \textit{normal filter} if it is closed under conjugation, namely if $H\in\sF$ and $\pi\in\sG$, then $\pi^{-1}H\pi\in\sF$ as well.

\begin{definition}
We say that $\tup{\PP,\sG,\sF}$ is a \textit{symmetric system} if $\PP$ is a notion of forcing, $\sG$ is an automorphism group of $\PP$, and $\sF$ is a normal filter of subgroups of $\sG$.

Let $\dot x$ be a $\PP$-name. Write $\sym_\sG(\dot x)=\{\pi\in\sG\mid\pi\dot x=\dot x\}$. We say that $\dot x$ is \textit{$\sF$-symmetric} if $\sym_\sG(\dot x)\in\sF$. We say that $\dot x$ is \textit{hereditarily $\sF$-symmetric} if this property is hereditary. Finally, $\HS_\sF$ is the class of all hereditarily $\sF$-symmetric names.
\end{definition}
We can notice that we do not really need $\sF$ to be a filter, but rather a filter base (satisfying the normality clause). We will judiciously ignore this, and work with filter bases as though they were filters. As $\sG$ and $\sF$ will usually be clear from context, we will omit them from the notation, wherever possible.
\medskip

\begin{theorem}[\cite{Jech:AC1973}, Theorem~5.14]
Suppose that $\tup{\PP,\sG,\sF}$ is a symmetric system and $G$ is a $V$-generic filter for $\PP$. Let $N=\HS^G=\{\dot x^G\mid\dot x\in\HS\}$, then $N$ is a transitive class of $V[G]$ satisfying $V\subseteq N$, and $N\models\ZF$. $N$ is called a \textit{symmetric extension of $V$}.\qed 
\end{theorem}

We also have a forcing relation $\forces^\HS$ which is defined as a relativization of the quantifiers and variables to $\HS$. It is not hard to check that $\forces^\HS$ is definable in $V$ (for each $\varphi$) and has the same forcing theorem as $\forces$. Namely $p\forces^\HS\varphi$ if and only if every symmetric extension generated by $G$ such that $p\in G$ satisfies $\varphi$. Moreover, if $\pi\in\sG$ used to define $\HS$, then we have a Symmetry Lemma for $\forces^\HS$.
\medskip

Finally, if $\sG$ is an automorphism group of $\PP$, we say that $\sG$ \textit{witnesses the homogeneity} of $\PP$ if for all $p,q\in\PP$ there is some $\pi\in\sG$ such that $\pi p$ is compatible with $q$. If $\tup{\PP,\sG,\sF}$ is a symmetric system such that $\sG$ witnesses the homogeneity of $\PP$, then we say that it is a \textit{homogeneous system}.

\subsection{Cohen forcing and its symmetric extension}
We are ready to begin with the first step of the construction of the Bristol model: adding the Cohen real and taking the first symmetric extension. For the remainder of the section we will assume $V=L$.

Let $\CC$ denote the Cohen forcing. Specifically, a condition in $\CC$ is a finite partial function from $\omega$ to $2$, it will be convenient to assume that the domain is any finite subset, rather than an initial segment. 

For $A\subseteq\omega$, let $\CC_A$ be the subforcing $\{p\in\CC\mid\dom p\subseteq A\}$. There is a canonical isomorphism between $\CC_A\times\CC_{\omega\setminus A}$ and $\CC$, given by $\tup{p,p'}=p\cup p'$. If $\pi$ is a permutation of $A$, it extends to a permutation of $\omega$ which is the identity on $\omega\setminus A$, and $\pi$ acts on $\CC$ by considering \[\pi p(\pi n)=p(n).\]

To establish the symmetric system we first need to talk about permutations of $\omega$ and $\omega_1$, and about almost disjoint families. 

\begin{definition}
We say that an almost disjoint family $\{A_\alpha\mid\alpha<\omega_1\}\subseteq[\omega]^\omega$ is a \textit{permutable family} if for every bounded $X\subseteq\omega_1$, there is a family $\{B_\alpha\mid\alpha\in X\}$ of pairwise disjoint sets such that $A_\alpha=^*B_\alpha$ for $\alpha\in X$, and for all $\xi<\omega_1$, $A_\xi\cap\bigcup\{B_\alpha\mid\alpha\in X\}$ is infinite if and only if $\xi\in X$.
\end{definition}

Note that the requirement that $A_\alpha=^*B_\alpha$ implies that $A_\xi\subseteq^*\bigcup\{B_\alpha\mid\alpha\in X\}$ for $\xi\in X$, and otherwise $A_\xi$ is almost disjoint from the union.

\begin{proposition}\label{prop:countable sep. family}
There exists a permutable family.
\end{proposition}
\begin{proof}
Let $\tup{T_\alpha\mid\alpha<\omega_1}$ be a strictly $\subseteq^*$-increasing sequence of subsets of $\omega$, and define $A_\alpha=T_{\alpha+1}\setminus T_\alpha$. Easily, $\{A_\alpha\mid\alpha<\omega_1\}$ is an almost disjoint family of sets. Suppose now that $X\subseteq\omega_1$ is a countable set, and let $\eta=\sup X+1$. Let $\{B_\alpha\mid\alpha<\eta\}$ be a refinement of the $A_\alpha$'s to a pairwise disjoint family of sets such that $B_\alpha\subseteq T_\eta$ for all $\alpha<\eta$. If $\xi\geq\eta$, then by the fact $T_\eta\subseteq^* T_\xi$, we get that $A_\xi$, namely $T_{\xi+1}\setminus T_\xi$, is almost disjoint from $T_\eta$ and therefore almost disjoint from $\bigcup\{B_\alpha\mid\alpha\in X\}$. If $\xi<\eta$, and $A_\xi\cap\bigcup\{B_\alpha\mid\alpha\in X\}$ is infinite, by the assumption that $A_\xi=^* B_\xi$ we get that it is necessarily the case that $\xi\in X$ as $\{B_\xi\mid \xi<\eta\}$ is a pairwise disjoint family.
\end{proof}

For the remainder of this section we fix a permutable family $\{A_\alpha\mid\alpha<\omega_1\}$.\medskip

Suppose that $\pi$ is a permutation of $\omega$ and $\Pi$ a permutation of $\omega_1$. We say that $\pi$ \textit{implements} $\Pi$ if for every $\alpha<\omega_1$, $\pi"A_\alpha=^* A_{\Pi(\alpha)}$.\footnote{This depends on the permutable family, of course.} We denote by $\iota(\pi)$ the permutation that $\pi$ implements, if it exists. It is easy to see that if $\pi_0$ and $\pi_1$ both implement permutations of $\omega_1$, then $\iota(\pi_0)\iota(\pi_1)=\iota(\pi_0\pi_1)$, and $\iota(\pi_0)^{-1}=\iota(\pi_0^{-1})$.

Let $\Pi$ be a permutation of $\omega_1$, we say that $\Pi$ is a \textit{bounded permutation}, if for some $\eta<\omega_1$, $\Pi$ is the identity above $\eta$. Namely, $\Pi$ is, for all practical purposes, a permutation of a bounded subset of $\omega_1$. We will say that $\eta$ is a domain of $\Pi$.

If $\{A_\alpha\mid\alpha\in I\}$ is a countable subfamily of our fixed permutable family, we will say that $\{B_\alpha\mid\alpha\in I\}$ is a \textit{disjoint approximation} if it is a family of pairwise disjoint sets, $A_\alpha=^* B_\alpha$ for all $\alpha\in I$, and it is satisfying the covering properties in the definition of a permutable family. If in addition $B_\alpha\subseteq A_\alpha$, then we will say it is a \textit{disjoint refinement}. We will often just say that $\cB$ is a disjoint approximation, or a disjoint refinement, if there is some countable $I\subseteq\omega_1$ such that $\cB$ is a disjoint approximation, or a disjoint refinement, of $\{A_\alpha\mid\alpha\in I\}$.
\begin{proposition}
Every bounded permutation of $\omega_1$ can be implemented.
\end{proposition}
\begin{proof}
Suppose that $\eta<\omega_1$ and $\Pi$ is a bounded permutation of $\omega_1$ and $\eta$ is a domain of $\Pi$. Then $\{A_\alpha\mid\alpha<\eta\}$ is a countable family of almost disjoint sets, therefore there exists a disjoint refinement $\{B_\alpha\mid\alpha<\eta\}$. If $\Pi(\alpha)=\beta$, define $\pi$ on $B_\alpha$ to be the unique order isomorphism from $B_\alpha$ into $B_\beta$; this defines $\pi$ on $\bigcup_{\alpha<\eta}B_\alpha$, and take $\pi$ to be the identity elsewhere. It is clear that $\pi$ is a permutation of $\omega$, and by the very definition of $\pi$ we have that $\pi"A_\alpha=^*\pi"B_\alpha=B_{\Pi(\alpha)}=^*A_{\Pi(\alpha)}$. So $\iota(\pi)=\Pi$.
\end{proof}

Let $\sG$ be the group of all the permutations of $\omega$ which implement a bounded permutation of $\omega_1$. Let $I\subseteq\omega_1$ be a countable set, and let $\cB=\{B_\alpha\mid\alpha\in I\}$ be a disjoint approximation. We define \[\fix(\cB)=\left\{\pi\in\sG\middd\pi\restriction\bigcup\cB=\id\right\},\] and we define $\sF$ to be the filter generated by \[\{\fix(\cB)\mid\cB\text{ is  a disjoint approximation}\}.\]

Let $\pi\in\sG$ and $\cB=\{B_\alpha\mid\alpha\in I\}$ be a disjoint approximation. Define $\cB'$ to be $\{\pi"B_\alpha\mid\alpha\in I\}$, then $\pi\fix(\cB)\pi^{-1}=\fix(\cB')$. Note that $\pi"B_\alpha=^*A_{\iota(\pi)(\alpha)}$ for $\alpha\in I$, so indeed $\cB'$ is a disjoint approximation. And therefore $\sF$ is normal.

\begin{remark}
This is a natural point to stop and ponder the nature of the definition of $\sF$. Why did we opt to take the more complicated definition using countable families of pairwise disjoint sets? After all, if $\cB=\{B_\alpha\mid\alpha\in I\}$ is such family, then $\pi\in\cB$ is simply some function which implements $\id$ when restricted to $I$.

Why, then, did we not choose to take the simpler definition where $\fix(I)$ is $\{\pi\mid\iota(\pi)\restriction I=\id\}$, and then take the filter generated by $\{\fix(I)\mid I\in[\omega_1]^{<\omega_1}\}$? The answer is that implementing a permutation is insensitive to finite changes. This means that $\fix(I)$ defined here will also include all the finitary permutations of $\omega$. This causes the filter to become irrelevant, since it means that no condition (other than the trivial condition) is fixed by a large group. And this causes the symmetric system to become trivial and ``kill'' any new sets.
\end{remark}

Let $\dot c$ denote the canonical name for the subset of $\omega$ obtained from the Cohen real, and for every $\alpha<\omega_1$ let $\dot c_\alpha$ denote the restriction of $\dot c$ to $A_\alpha$ defined as $\{\tup{p,\check n}\mid p(n)=1\land\dom p\subseteq A_\alpha\}$. We shall omit the dot when referring to the interpretation of these names in the extension.
\begin{theorem}
Let $M$ be the symmetric extension obtained by the symmetric system $\tup{\CC,\sG,\sF}$. Then for every $\alpha<\omega_1$ we have that $c_\alpha\in M$, and since $M\models\lnot\AC$ it follows that $c\notin M$.
\end{theorem}
\begin{proof}
To obtain that $c_\alpha\in M$ it is enough to show that $\dot c_\alpha$ is a symmetric name. However, taking $\cB$ to be $\{A_\alpha\}$, this is a disjoint approximation of itself, and of course that $\fix(\{A_\alpha\})\leq\sym(\dot c_\alpha)$.

To see that the axiom of choice fails in $M$, suppose that $\dot f\in\HS$ such that $p\forces^\HS\dot f\colon\power(\check\omega)\to\check\theta$, for some ordinal $\theta$ (note that we use $\forces^\HS$, which means that we are only considering the reals in $M$). Let $\cB$ be a countable family which is a support for $\dot f$. Pick some $\eta$ such that $A_\eta\nsubseteq^*\bigcup\cB$, and let $q\leq p$ such that $q\forces\dot f(\dot c_\eta)=\check\xi$. Pick $m,n\in\omega$ such that the following holds:
\begin{enumerate}
\item $m,n\notin\bigcup\cB\cup\dom q$.
\item Exactly one of $m$ and $n$ is an element of $A_\eta$. Without loss of generality, $n\in A_\eta$.
\end{enumerate}
Consider now the automorphism which is the $2$-cycle $(m\ n)$. It is clearly going to be in $\fix(\cB)$, and $\pi q=q$, but this translates to $q\forces\dot f(\pi\dot c_\eta)=\check\xi=f(\dot c_\eta)$. However there is some $q'\leq q$ such that $q'\forces\pi\dot c_\eta\neq\dot c_\eta$, such $q'$ forces that $\dot f$ is not injective, so $q$ cannot force that as well.

Therefore in $M$ the reals cannot be well-ordered. Consequently, $c\notin M$, as in that case $L[c]\subseteq M\subseteq L[c]$ which is a contradiction to the fact that $M\models\lnot\AC$.
\end{proof}
\subsection{Forcing over the symmetric extension}
We continue to work in the same setting as before. We want to define a forcing in $M$ and find it a generic in $L[c]$. Moreover, we would like this forcing to be sufficiently distributive so it does not add any new reals. This will ensure that we did not accidentally add $c$ back into the model. We also argue that in that case, we did not force the axiom of choice back, as the following lemma shows.
\begin{lemma}
Suppose that $W_0\subseteq L[c]$ is a model where the axiom of choice fails, and $W_1\subseteq L[c]$ is a generic extension of $W_0$ such that $\power(\omega)^{W_0}=\power(\omega)^{W_1}$, then the axiom of choice fails in $W_1$ as well.
\end{lemma}
\begin{proof}
Let $X\in W_0$ be a transitive set which cannot be well-ordered. If $W_1\subseteq L[c]$ is any model where $X$ can be well-ordered, then there is a set of ordinals $A$ encoding the $\in$ relation of $X\cup\{X\}$. However this set of ordinals lies in $L[c]$, therefore $L[A]$ is a model of $\ZFC$ intermediate to $L[c]$. It follows by the intermediate model theorem \cite[Lemma~15.43]{Jech:ST2003} that there is a Cohen real $r$ such that $L[A]=L[r]$.

Since $W_0$ and $W_1$ have the same reals, if $X$ cannot be well-ordered in $W_0$, it cannot be well-ordered in $W_1$ as well.
\end{proof}
Define for every $\alpha<\omega_1$ the set $R_\alpha$ as $\power(\omega)^{L[c_\alpha]}$. We want to give $R_\alpha$ a name. While the obvious name would be to take all the canonical $\CC_{A_\alpha}$-names for reals, we opt for something slightly different which will make our lives easier down the road.
\[\dot R_\alpha=\left\{\dot x\middd\begin{array}{l}
\dot x\in\HS\text{ is a canonical name for a real, and }\\
\exists B=^*A_\alpha\text{ such that }\dot x\text{ is a }\CC_B\text{-name}
\end{array}\right\}^\bullet.\]
While this name is indeed more complicated, it has the benefit of the following proposition being true.
\begin{proposition}\label{prop:nice-up-sym}
For every $\pi\in\sG$, $\pi\dot R_\alpha=\dot R_{\iota(\pi)(\alpha)}$.\qed
\end{proposition}

Therefore if $\iota(\pi)(\alpha)=\alpha$, then $\pi\dot R_\alpha=\dot R_\alpha$; thus each $\dot R_\alpha\in\HS$.

\begin{proposition}
Suppose that $A\in[\omega_1]^{<\omega_1}\cap L$. Then $\tup{R_\alpha\mid\alpha\in A}\in M$.
\end{proposition}
\begin{proof}
It is enough to show that $\tup{\dot R_\alpha\mid\alpha\in A}^\bullet\in\HS$. However, this is nearly trivial. If $\cB$ is any disjoint approximation of $\{A_\alpha\mid\alpha\in A\}$, then $\pi\dot R_\alpha=\dot R_\alpha$ whenever $\pi\in\fix(\cB)$.
\end{proof}
Define $\dot\sigma$ to be $\tup{\dot R_\alpha\mid\alpha<\omega_1}^\bullet$, and for $A\subseteq\omega_1$, let $\dot\sigma_A$ denote the restriction of $\dot\sigma$ to $A$. The above proposition shows that if $A$ is countable, then $\dot\sigma_A\in\HS$. We define $\dot\QQ=\{\pi\dot\sigma_A\mid A\in[\omega_1]^{<\omega_1}\land\pi\in\sG\}^\bullet$, and we define the order of $\dot\QQ$ as reverse inclusion. Clearly $\dot\QQ\in\HS$, as it is stable under every automorphism (the same goes for the order of $\dot\QQ$).

Note that by \autoref{prop:nice-up-sym}, $\pi\in\sG$ acts on $\dot\sigma$ by permuting the range of the sequence. What $\QQ$ is doing, is trying to well-order the set $\{\power(\omega)^{L[c_\alpha]}\mid\alpha<\omega_1\}$, which lies in $M$ (the indexation by $\omega_1$ is not symmetric, of course, this is $\sigma$ itself). It also follows that if $\pi$ and $\tau$ are two permutations such that $\iota(\pi)$ and $\iota(\tau)$ act the same on $A$, then $\tau\dot\sigma_A=\pi\dot\sigma_A$. The following theorem is the main part of this section. 

\begin{theorem}\label{thm:second forcing}
The sequence $\sigma$ is $M$-generic for $\QQ$, and $M[\sigma]$ has the same reals as $M$.
\end{theorem}

The rest of this section will be devoted for the proof of the theorem. We begin with a key lemma.
\begin{lemma}\label{lemma:initial segment}
Suppose that $\dot D\in\HS$ and $p\forces\dot D\subseteq\dot\QQ$ is a dense open set. There is some $\eta<\omega_1$ such that for every $\pi$ and $A$ such that $p\forces\pi\dot\sigma_A\in\dot D$ and $\eta\subseteq A$, if $\tau\dot\sigma_A$ is a condition such that $p\forces\pi\dot\sigma_\eta=\tau\dot\sigma_\eta$, then $p\forces\tau\dot\sigma_A\in\dot D$ as well.
\end{lemma}
In other words, the lemma states that if $D$ is a dense open set, then there is a crucial bit of information---$\dot\sigma_\eta$ and its orbit under $\sym(\dot D)$---so that being a member of $D$ essentially depends on that information.
\begin{proof}
Let $\eta$ be such that for some disjoint approximation $\cB=\{B_\alpha\mid\alpha<\eta\}$, $\fix(\cB)\leq\sym(\dot D)$. We claim that this $\eta$ satisfies the wanted condition. Suppose that $\tau$ is a permutation as in the assumptions, then $\iota(\tau)$ and $\iota(\pi)$ agree on the ordinals up to $\eta$. We can find a permutation $\pi'$ such that $\iota(\pi')=\iota(\tau\pi^{-1})$, and $\pi'\in\fix(\cB)$ as well $\pi' p=p$. This completes the proof, of course, since now we get that $\pi' p\forces\pi'(\pi\dot\sigma_A)\in\pi'\dot D$ which gets translated to $p\forces\tau\dot\sigma_A\in\dot D$.
\end{proof}
\begin{corollary}
$\sigma$ is $M$-generic for $\QQ$.
\end{corollary}
\begin{proof}
Suppose that $\dot D$ is a name for a dense open set as in the lemma. Let $p$ be a condition and $\pi\dot\sigma_A$ such that $\iota(\pi)\restriction\eta=\id$, so $p\forces\pi\dot\sigma_A\leq_{\dot\QQ}\dot\sigma_\eta$, and $p\forces\pi\dot\sigma_A\in\dot D$. By the lemma we get that $p\forces\dot\sigma_A\in\dot D$.
\end{proof}
\begin{corollary}\label{cor:baby reals}
In $M$ the forcing $\QQ$ is ${\leq}\omega$-distributive, namely if $\tup{D_n\mid n<\omega}$ is a sequence of dense open sets, then $D=\bigcap_{n<\omega}D_n$ is a dense open set.
\end{corollary}
\begin{proof}
Suppose that $\dot f\in\HS$ is such that $p\forces\dot f(\check n)=\dot D_n$ is a dense open set for every $n<\omega$. It is not hard to check that if $\cB=\{B_\alpha\mid\alpha<\eta\}$ is such that $\fix(\cB)\leq\sym(\dot f)$, then this is also the case that $\fix(\cB)\leq\sym(\dot D_n)$ for all $n$. Suppose that $\eta\subseteq A$ and $\pi$ is some permutation in $\sG$. For every $n$, let $\{p_{n,m}\mid m<\omega\}$ be a maximal antichain below $p$ such that there is some ordinal $\xi_{n,m}$ with $A\subseteq\xi_{n,m}$ and some permutation $\tau_{n,m}$ such that $p_{n,m}\forces\tau_{n,m}\dot\sigma_{\xi_{n,m}}\in\dot D_n$ and $\tau_{n,m}\dot\sigma_{\xi_{n,m}}\leq_{\dot\QQ}\pi\dot\sigma_A$. 

By the lemma, it follows that $p_{n,m}$ actually forces that every extension of $\pi\dot\sigma_A$ whose domain is $\xi_{n,m}$ will be in $\dot D_n$. Let $\xi_n=\sup\{\xi_{n,m}\mid m<\omega\}$, therefore $p$ forces that every extension of $\pi\dot\sigma_A$ with domain $\xi_n$ lies in $\dot D_n$. Finally taking $\xi=\sup\{\xi_n\mid n<\omega\}$ will satisfy that if $\tau\dot\sigma_\xi$ is any condition extending $\pi\dot\sigma_A$, then for all $n<\omega$, $p\forces\tau\dot\sigma_\xi\in\dot D_n$. Therefore the intersection of all the $D_n$'s is indeed dense as wanted.
\end{proof}

This completes the proof of the theorem. The next step would be to pick an permutable family of size $\aleph_2$ of subsets of $\omega_1$, and repeat the process as we did here. But instead of working one step at at time, we will instead switch gears into high action mode. We will cover some of the technical tools needed for this construction, and then describe the general structure of the proof in one fell swoop. We end this with a remark that while it might seem that $\QQ$ ``should'' restore the Cohen real, and that we want to somehow go back-and-forth to $L[c]$ and down into inner models, this is not the case. In order to restore the Cohen real we need to be able and choose the $c_\alpha$'s themselves, as generators for each $R_\alpha$. What $\QQ$ did was only to well-order the set of the $R_\alpha$'s. Of course, the next step would be to forget this well-ordering, and remember only fragments of it via a symmetric extension, and continue ad infinitum.
\begin{remark}
We could have proved \autoref{thm:second forcing} using the following argument: in $L[c]$ the forcing $\QQ$ is isomorphic to $\Add(\omega_1,1)^L$ (the isomorphism is not in $M$, of course). Therefore $\QQ$ is ${\leq}\omega$-distributive in $L[c]$, and this is absolute to $M$.

Of course, we would still have to show that $\sigma$ is $M$-generic for $\QQ$, which would require a proof more or less along the lines of \autoref{lemma:initial segment}, although in simpler form.
\end{remark}
\section{Tools for your Bristol construction}
We have covered the basic details of one symmetric extension; but we need an apparatus for iterating them. We will provide a brief description of the method for iterating symmetric extensions. After this we will deal with the generalizations of permutable families, and permutable scales which will be needed for pushing the construction through the limit steps.
\subsection{Productive iterations of symmetric extensions}
Iteration of symmetric extensions is a framework for (as the name suggests) iterating symmetric extensions, i.e.\ a symmetric extension of a symmetric extension, and so on. While we can do this by hand for finitely many steps, the framework does offer a method which extends transfinitely as well. The full details of the construction can be found in \cite{Karagila:2016}, and we will only cover a small part of it necessary for this work. We will not prove any statements here, and some of them will be slightly modified to accommodate the Bristol model construction.

\begin{definition}Suppose that $\PP$ is a forcing notion, $\pi$ is an automorphism of $\PP$, and $\dot A$ is a $\PP$-name. We say that $\pi$ \textit{respects} $\dot A$ if $\forces_\PP\dot A=\pi\dot A$. If $\dot A$ carries an implicit structure (e.g.\ a forcing notion) then we require that the implicit structure is respected as well.
\end{definition}

Let us consider a two-step iteration of two symmetric systems, $\tup{\QQ_0,\sG_0,\sF_0}$ and $\tup{\dot\QQ_1,\dot\sG_1,\dot\sF_1}^\bullet$. If $\PP$ is the two-step iteration $\QQ_0\ast\dot\QQ_1$, we would like to be able and isolate $\PP$ names which are equivalent to $\QQ_0$-names which are in $\HS_{\sF_0}$, and they are themselves names for $\dot\QQ_1$-names which will also be symmetric. In order for automorphisms from $\sG_0$ to even have a chance to preserve the property of being a $\dot\QQ_1$-name, an automorphism in $\dot\sG_1$, or anything else definable from the symmetric system $\tup{\QQ_1,\dot\sG_1,\dot\sF_1}^\bullet$ we need to require that all the automorphisms in $\sG_0$ respect the symmetric system.

\begin{definition}\label{def:semi-direct-two-steps}
Let $\tup{\QQ_0,\sG_0,\sF_0}$ be a symmetric system, and $\tup{\dot\QQ_1,\dot\sG_1,\dot\sF_1}^\bullet$ be a name for a symmetric system in $\HS_{\sF_0}$ which is respected by all the automorphisms in $\sG_0$. Let $\PP$ denote the iteration $\QQ_0\ast\dot\QQ_1$.
\begin{enumerate}
\item For $\pi\in\sG_0$, define $\gaut{\pi}\tup{q_0,\dot q_1}=\tup{\pi q_0,\pi(\dot q_0)}$. This is well-defined as $\pi$ acts on the $\QQ_0$-names, and it respects $\dot\QQ_1$. 
\item For $\dot\sigma$ such that $\forces_{\QQ_0}\dot\sigma\in\dot\sG_1$, define $\gaut{\dot\sigma}\tup{q_0,\dot q_1}=\tup{q_0,\dot\sigma(\dot q_1)}$, where $\dot\sigma(\dot q_1)$ is a $\QQ_0$-name for a condition $\dot q'$ in $\dot\QQ_1$ satisfying $\forces_{\QQ_0}\dot q'=\dot\sigma(\dot q_1)$.
\item If $\pi\in\sG_0$ and $\dot\sigma$ is a name such that $\forces_{\QQ_0}\dot\sigma\in\dot\sG_1$, we define: \[\gaut{\tup{\pi,\dot\sigma}}\tup{q_0,\dot q_1}=\gaut{\pi}\gaut{\dot\sigma}\tup{q_0,\dot q_1}=\tup{\pi q_0,\pi(\dot\sigma(\dot q_1))}=\tup{\pi q_0,\pi(\dot\sigma)(\pi\dot q_1)}.\]
\end{enumerate}
\end{definition}
We denote by $\sG_0\ast\dot\sG_1$, and usually by $\cG_1$, the group of all automorphisms of the form that appears in clause (3) of the above definition. This is indeed a group of automorphisms. Upon close inspection there are similarities between the action of $\cG_1$ on the iteration and the semi-direct product of $\sG_0$ and $\dot\sG_1$, and indeed $\cG_1$ is called the generic semi-direct product. It is worth pointing out that the order $\gaut{\pi}\gaut{\dot\sigma}$ can be reversed by paying a ``price'' of conjugating $\dot\sigma$ by $\pi$. This is evident from the following proposition.

\begin{proposition}\label{prop:two-step-conj}
Under the assumptions of \autoref{def:semi-direct-two-steps}, let $\tup{\pi_0,\dot\sigma_0}$ and $\tup{\pi_1,\dot\sigma_1}$ be two pairs of automorphisms. Then the following hold:
\begin{enumerate}
\item $\gaut{\tup{\pi_1,\dot\sigma_1}}\gaut{\tup{\pi_0,\dot\sigma_0}}=\gaut{\tup{\pi_1\pi_0,\pi_0^{-1}(\dot\sigma_1)\dot\sigma_0}}$,
\item $\gaut{\tup{\pi_1,\dot\sigma_1}}^{-1}=\gaut{\tup{\pi_1^{-1},\pi_1(\dot\sigma_1^{-1})}}$.
\end{enumerate}
\end{proposition}

\begin{definition}
Under the assumptions of \autoref{def:semi-direct-two-steps}, suppose that $\dot x$ is a $\QQ_0\ast\dot\QQ_1$-name. We say that $\dot x$ is \textit{$\cF_1$-respected} if there exists $H_0\in\sF_0$ and $\dot H_1$ such that $\forces_{\QQ_0}\dot H_1\in\dot\sF_1$, and there is a pre-dense set $D$ such that for $p\in D$, if $\tup{\pi,\dot\sigma}$ such that $\pi\in H_0$ and $p\forces_{\QQ_0}\dot\sigma\in\dot H_1$, then $p\forces_{\QQ_0\ast\dot\QQ_1}\gaut{\tup{\pi,\dot\sigma}}\dot x=\dot x$. We say that $\dot x$ is a \textit{hereditarily $\cF_1$-respected} if it is respected, and every name $\dot y$ which appears in $\dot x$ is hereditarily $\cF_1$-respected.
\end{definition}

We call $\tup{H_0,\dot H_1}$ a \textit{support} when $H_0\in\sF_0$ and $\dot H_1$ is a $\QQ_0$-name such that $\forces_{\QQ_0}\dot H_1\in\dot\sF_1$. Supports will have a more significant role when the iteration is long. In general, it seems that one would like to define a support as a group in some filter of subgroups on $\cG_1$ which satisfies some properties. There is an apparent difficulty when trying to show that this filter is a normal filter of subgroups, as discussed in \cite[\S4]{Karagila:2016}. This definition, however, is adequate, as shown in \S5 of the same paper.

\begin{theorem}
Suppose that $\tup{\QQ_0,\sG_0,\sF_0}$ is a symmetric system and $\tup{\dot\QQ_1,\dot\sG_1,\dot\sF_1}^\bullet$ is a name for a symmetric system which is respected by $\sG_0$. Then the class of $\QQ_0\ast\dot\QQ_1$-names which are hereditarily $\cF_1$-respected is exactly the class of names which are generically equal to the symmetric iteration obtained by $\tup{\QQ_1,\sG_1,\sF_1}$ from the symmetric iteration obtained by $\tup{\QQ_0,\sG_0,\sF_0}$.
\end{theorem}

We can generalize this to any finite support iteration now, requiring that each iterand is respected by all the previous automorphisms. But before we do that, we should point out a genericity issue. The definition of $\gaut{\dot\sigma}$ utilizes---quite heavily---the fact that we can mix names, as we assume the axiom of choice. This requires use a filter which is $V$-generic for $\QQ_0\ast\dot\QQ_1$. This is a big problem if we want to use this method to construct the Bristol model, since we want to pick our generics by hand, and ensure they come out of the Cohen real. To solve this problem, we define in \cite[\S8]{Karagila:2016} the notion of a \textit{productive iteration}, which means that we are doing something very akin to a product in its ``ground model canonicity''. Namely, a productive iteration is an iteration of symmetric extension such that in the full generic extension, each iterand is isomorphic to a symmetric system in the ground model, but the isomorphism itself is not present in the intermediate model itself.

\begin{definition}
Suppose that $\tup{\QQ_0,\sG_0,\sF_0}$ is a symmetric system and $\tup{\dot\QQ_1,\dot\sG_1,\dot\sF_1}^\bullet$ is a name for a symmetric system. We say that $\tup{\QQ_0\ast\dot\QQ_1,\cG_1,\cF_1}$ is a \textit{two-step productive symmetric iteration} if these six conditions hold:
\begin{enumerate}
\item The name $\tup{\dot\QQ_1,\dot\sG_1,\dot\sF_1}^\bullet$ is respected by all the automorphisms in $\sG_0$.
\item The names $\dot\QQ_1$, $\dot\sG_1$ and $\dot\sF_1$ are $\bullet$-names which are hereditarily respected on a group in $\sF_0$.\footnote{We implicitly require that also the order structure and group operation are $\bullet$-names.}
\item The conditions in $\QQ_0\ast\dot\QQ_1$ are exactly pairs $\tup{q_0,\dot q_1}$ such that $q_0\in\QQ_0$ and $\dot q_1$ appears in $\dot\QQ_1$.
\item For every $\dot\sigma$ and $\dot\pi$ which appear in $\dot\sG_1$, $1_{\QQ_0}$ decides the truth of $\dot\sigma=\dot\pi$.
\item For every $\dot q$ which appears in $\dot\QQ_1$ and $\dot\sigma$ which appears in $\dot\sG_1$, there is some $\dot q'$ which appears in $\dot\QQ_1$ such that $\forces_{\QQ_0}\dot q'=\dot\sigma(\dot q)$.
\item Every automorphism in $\cG_1$ is of the form $\gaut{\tup{\pi_0,\dot\pi_1}}$ where $\pi_0\in\sG_0$ and $\dot\pi_1$ appears in $\dot\sG_1$.
\end{enumerate}
\end{definition}
The idea is that $\dot\QQ_1$ is a collection of some names which are ``more or less in $\HS_{\sF_0}$'', with $\dot\sG_1$ a copy of an automorphism group of these names, and $\dot\sF_1$ a copy of a filter of subgroups from the ground model. There is something to be said about the definition of respected names, but this will only affect the general case, not the two-step iterations.\medskip

The following definitions will be used to generalize the two-step iteration, but to avoid excessive terminology, we will omit some assumptions from them. The definitions are meant to be read in the context of \autoref{def:prod-itr}, and to consider the two-step case as a bootstrapping definition when needed.\medskip
p
If $\tup{\dot\QQ_\alpha\mid\alpha<\delta}$ defines a finite-support iteration such that for all $\alpha<\delta$, $\dot\QQ_\alpha$ is a $\bullet$-name, we will always consider the iteration poset $\PP_\alpha$ defined as the sequences $p$ such that $p(\alpha)$ appears in $\dot\QQ_\alpha$. All our iterations will have this form, and we will always use $\PP$'s to denote the finite-support iteration of the $\dot\QQ$'s.

\begin{definition}\label{def:gen-semi-dir}
Suppose that $\tup{\dot\QQ_\alpha,\dot\sG_\alpha\mid\alpha<\delta}$ is a finite-support iteration satisfying
\begin{enumerate}
\item Every $\dot\sG_\alpha$ is a $\bullet$-name for an automorphism group of $\dot\QQ_\alpha$, such that $1_\alpha$ decides $\dot\pi=\dot\sigma$ for any $\dot\pi,\dot\sigma$ which appear in $\dot\sG_\alpha$.
\item For all $\dot q$ appearing in $\dot\QQ_\alpha$ and $\dot\pi$ appearing in $\dot\sG_\alpha$, there is some $\dot q'$ appearing in $\dot\QQ_\alpha$ such that $\forces_\alpha\dot q'=\dot\pi(\dot q)$.
\item For all $\alpha$, every automorphism in $\cG_\alpha$ respects both $\dot\QQ_\alpha$ and $\dot\sG_\alpha$.
\end{enumerate}
Suppose that $\cG_\alpha$ was defined for all $\alpha<\delta$ as an automorphism group of $\PP_\alpha$. We define $\cG_\delta$ in the following way:
\begin{itemize}
\item For $\delta=0$, $\cG_0$ is the trivial group.
\item For $\delta=\alpha+1$, $\cG_\delta=\cG_\alpha\ast\dot\sG_\alpha$.
\item For $\delta$ limit, $\cG_\delta$ is the direct limit of $\cG_\alpha$ for $\alpha<\delta$.
\end{itemize}
\end{definition}
\begin{proposition}
Assume the conditions of the previous definition. An automorphism in $\cG_\delta$ is exactly one of the form $\gaut{\vec\pi}$ such that:
\begin{enumerate}
\item $\vec\pi$ is a sequence $\tup{\dot\pi_\alpha\mid\alpha<\delta}$ where $\dot\pi_\alpha$ appears in $\dot\sG_\alpha$.
\item For all but finitely many $\alpha<\delta$, $\dot\pi_\alpha$ is the forced to identity function.
\end{enumerate}
Suppose that $\vec\pi$ is such sequence and $C(\vec\pi)=\{\alpha<\delta\mid\nforces_\alpha\dot\pi_\alpha=\id^\bullet\}$, then $\gaut{\vec\pi}p$ is given recursively: if $\alpha=\max C(\vec\pi)$, then $\gaut{\vec\pi}p=\gaut{\vec\pi\restriction\alpha}\gaut{\dot\pi_\alpha}p$, where \[\gaut{\dot\pi_\alpha}p=p\restriction\alpha^\smallfrown\dot\pi_\alpha(p(\alpha))^\smallfrown\dot\pi_\alpha(p\restriction(\alpha,\delta)).\]
\end{proposition}
In other words, we apply $\gaut{\vec\pi}$ on $p$ by breaking $p$ into the finite intervals defined by $C(\vec\pi)$, then working from the maximum coordinate downwards, and each step applying $\dot\pi_\alpha$ on the relevant coordinate as an automorphism of the $\alpha$th iterand, and on the part above $\alpha$ as an automorphism acting on the $\PP_{\alpha+1}$-name. We have a generalization of \autoref{prop:two-step-conj} as well.
\begin{proposition}
Suppose that $\vec\pi$ and $\vec\sigma$ are in $\cG_\delta$, then:
\begin{enumerate}
\item The sequence defining $\gaut{\vec\pi}^{-1}$, denoted by $\vec\pi^{-1}$, is given by $\tup{\gaut{\vec\pi}\dot\pi_\alpha^{-1}\mid\alpha<\delta}$.
\item $\gaut{\vec\pi}\gaut{\vec\sigma}$ is given by the sequence of automorphisms $\tup{(\gaut{\vec\sigma^{-1}}\dot\pi_\alpha)\dot\sigma_\alpha\mid\alpha<\delta}$.
\end{enumerate}
\end{proposition}
It should be remarked that the requirements for the iterations are highly nontrivial. Generally, iterations of weakly homogeneous forcings need not be weakly homogeneous, however the following proposition shows this is not the case in the symmetric case.
\begin{proposition}
Suppose that $\tup{\dot\QQ_\alpha,\dot\sG_\alpha\mid\alpha<\delta}$ are as in \autoref{def:gen-semi-dir}, and for all $\alpha<\delta$, $\forces_\alpha\dot\sG_\alpha$ witnesses the homogeneity of $\dot\QQ_\alpha$. Then $\cG_\delta$ witnesses the homogeneity of $\PP_\delta$.
\end{proposition}

\begin{definition}\label{def:supports}
Suppose that $\tup{\dot\QQ_\alpha,\dot\sG_\alpha,\dot\sF_\alpha\mid\alpha<\delta}$ is such that for all $\alpha<\delta$,
\begin{enumerate}
\item The assumptions of \autoref{def:gen-semi-dir} hold, and in addition $\dot\sF_\alpha$ is a $\bullet$-name such that $\forces_\alpha\dot\sF_\alpha$ is a normal filter of subgroups on $\dot\sG_\alpha$.
\item Every automorphism of $\cG_\alpha$ respects $\dot\sF_\alpha$.
\item For every $\alpha$, and every $\dot H$, $1_\alpha$ decides $\dot H\in\dot\sF_\alpha$.
\end{enumerate}
We say that $\vec H$ is an \textit{excellent $\cF_\delta$-support} if it is a sequence $\tup{\dot H_\alpha\mid\alpha<\delta}$ such that:
\begin{enumerate}
\item $\dot H_\alpha$ appears in $\dot\sF_\alpha$.
\item For all but finitely many $\alpha$, $\dot H_\alpha=\dot\sG_\alpha$.
\end{enumerate}
We define $\cF_\delta$ to be the set of all $\cF_\delta$-supports.
\end{definition}
We will write $p\forces_\delta\vec\pi\in\vec H$ if for every $\alpha<\delta$, $p\forces_\delta\dot\pi_\alpha\in\dot H_\alpha$; similarly we will write $\vec H\cap\vec K$ or $\vec H\subseteq\vec K$ for pointwise intersection, inclusion, and so on. We use $C(\vec H)$ to denote the set $\{\alpha<\delta\mid\dot H_\alpha\neq\dot\sG_\alpha\}$.
 
\begin{remark}
The reason that we say that $\vec H$ is an \textit{excellent} support, and not just any support, is that in the non-productive context we weaken the requirements, and only require that $\forces_\delta\{\check\alpha\mid\dot H_\alpha\neq\dot\sG_\alpha\}$ is finite. In particular, we do not require it to have any specific size or bound. This allows us to use the mixing lemma more easily when defining everything. Of course, that would again require the generic filter to be generic for the iteration, and here we want to weaken this substantially.
\end{remark}

\begin{definition}
Under the assumptions of \autoref{def:supports}, we say that a $\PP_\delta$-name $\dot x$ is \textit{$\cF_\delta$-respected} if there exists an excellent support $\vec H$ such that for every $\vec\pi\in\vec H$, $\forces_\delta\gaut{\vec\pi}\dot x=\dot x$. We say that $\dot x$ is \textit{hereditarily $\cF_\delta$-respected}, if it is $\cF_\delta$-respected, and every $\dot y$ which appears in $\dot x$ is hereditarily $\cF_\delta$-respected.
\end{definition}
Comparing this to the two-step definition of $\cF_1$-respected, the absence of a predense set is noted. This is because in the context of productive iterations, we require that $1_\alpha$ decides enough information to ensure the pre-dense set is $\{1_\alpha\}$.

\begin{definition}\label{def:prod-itr}
We say that $\tup{\dot\QQ_\alpha,\dot\sG_\alpha,\dot\sF_\alpha\mid\alpha<\delta}$ is a \textit{productive iteration} of symmetric extensions if it satisfies the assumptions of \autoref{def:supports} and for all $\alpha<\delta$, $\tup{\dot\QQ_\alpha,\dot\sG_\alpha,\dot\sF_\alpha}^\bullet$ is hereditarily $\cF_\alpha$-respected name. 
\end{definition}
We use $\IS_\alpha$ to denote the class of hereditarily $\cF_\delta$-respected names. 
\begin{definition}
Suppose that $\tup{\dot\QQ_\alpha,\dot\sG_\alpha,\dot\sF_\alpha\mid\alpha<\delta}$ is a productive iteration, we define $p\forces^\IS_\delta\varphi(\dot x_1,\ldots,\dot x_n)$ as $p\forces_\delta\varphi^{\IS_\delta}(\dot x_1,\ldots,\dot x_n)$. Namely, we relativize the quantifiers in $\varphi$ to $\IS_\delta$ and require that $\dot x_i\in\dot\IS_\delta$.
\end{definition}

\begin{definition}
Suppose that $\PP_\delta$ is a productive iteration. 
\begin{enumerate}
\item
We say that $D\subseteq\PP_\delta$ is a \textit{symmetrically dense open} set if there is some excellent support $\vec H$, such that for all $\vec\pi\in\vec H$ and $p\in D$, $\gaut{\vec\pi}p\in D$.
\item
We say that $G\subseteq\PP_\delta$ is a \textit{symmetrically $V$-generic filter} if for all symmetrically dense open sets $D\in V$, $D\cap G\neq\varnothing$.
\end{enumerate}
\end{definition}

\begin{lemma}\label{lemma:decision sets are symmetrically dense}
If $\dot x\in\IS_\delta$ and $\varphi$ is a statement in the language of forcing, then $\{p\in\PP_\delta\mid p\forces_\delta^\IS\varphi(\dot x)\lor p\forces_\delta^\IS\lnot\varphi(\dot x)\}$ contains a symmetrically dense open set.
\end{lemma}

\begin{theorem}
Suppose that $\PP_\delta$ is a productive iteration, $\varphi(\dot x)$ is a formula in the language of forcing, then the following are equivalent:
\begin{enumerate}
\item $p\forces^\IS_\delta\varphi(\dot x)$.
\item For all symmetrically $V$-generic filter $G$ such that $p\in G$, $\IS_\delta^G\models\varphi(\dot x^G)$.
\item For all $V$-generic filter $G$ such that $p\in G$, $\IS_\delta^G\models\varphi(\dot x^G)$.
\end{enumerate}
\end{theorem}

Finite support iterations, however, have some problems: Cohen reals tend to ``pop up'' at limit steps, and cardinals are collapsed if we are not careful to use ccc iterands. This is generally unwanted, especially if we want to iterate the construction through the class of ordinals, and as shown in the previous section, the second iterand is not going to be a ccc poset. However, the second iterand does not add reals, and so we have the following preservation theorem which will help us overcome this barrier.
\begin{theorem}
Suppose that $\tup{\dot\QQ_\alpha,\dot\sG_\alpha,\dot\sF_\alpha\mid\alpha<\delta}$ is a productive iteration such that for all $\alpha<\delta$, $\forces_\alpha\dot\sG_\alpha$ witnesses the homogeneity of $\dot\QQ_\alpha$. If $\eta$ is an ordinal such that there exists $\alpha<\delta$, so for all $\beta>\alpha$, $\forces^\IS_\beta\tup{\dot\QQ_\beta,\dot\sG_\beta,\dot\sF_\beta}^\bullet$ does not add sets of rank $<\check\eta$. Then for all $\dot x\in\IS_\delta$, if $p\forces_\delta^\IS\dot x$ has rank $<\check\eta$, then there is some $q\leq_\delta p$ and $\dot y\in\IS_\alpha$ such that $q\forces_\delta^\IS\dot x=\dot y$.
\end{theorem}
In other words, if at some point we stop adding sets of rank $\eta$, then under the assumption of homogeneity in each step, we do not add sets of rank $\eta$ at limit steps either. In particular no Cohen reals are added, and cardinals are preserved. The above theorem can be stated for any filtration of the universe which has a robust definition, specifically this will be used when we talk about $\alpha$-sets in \autoref{thm:alpha sets}. We also get the following theorem as a corollary.
\begin{theorem}
Suppose that $\tup{\dot\QQ_\alpha,\dot\sG_\alpha,\dot\sF_\alpha\mid\alpha\in\Ord}$ is a class-length productive iteration such that for all $\alpha$, $\forces_\alpha^\IS\tup{\dot\QQ_\alpha,\dot\sG_\alpha,\dot\sF_\alpha}^\bullet$ is a homogeneous system. Moreover, suppose that for all $\eta$ there is some $\alpha$, such that no sets of rank $\eta$ are added after the $\alpha$th iterand. Then for any symmetrically $V$-generic $G$, $\IS^G$ is a model of $\ZF$.
\end{theorem}

\begin{remark}
One of the central notions in \cite{Karagila:2016} is tenacity. Tenacity means that we can always assume that a condition is not moved by a large group of automorphisms. As it turns out, every symmetric system can be replaced by one which satisfies tenacity; nevertheless tenacity assumptions make the general construction easier. In this work, however, the notion is not particularly needed, as it arises naturally from the definitions of the symmetric systems. We will safely ignore it here.
\end{remark}

Finally, we will need one last definition about iterations, which we essentially saw in the baby step construction.
\begin{definition}We say that $\PP\ast\dot\QQ$ is an \textit{upwards homogeneous} iteration, if for all $\tup{p,\dot q},\tup{p,\dot q'}$ there is some $\pi\in\aut(\PP)$ such that $\pi p=p$ and $p\forces_\PP\pi(\dot q)$ is compatible with $\dot q'$. If $\PP$ is part of a symmetric system, we require that $\pi$ comes from the relevant automorphism group.
\end{definition}
\subsection{Permutable families}
Each step of the iteration will introduce a sequence of an appropriate length, generalizing the baby step details we saw. We will need two types of structures to handle these constructions. permutable families for successor steps, and permutable scales for limit steps.
\begin{definition}
Let $\kappa$ be any infinite cardinal. We say that an almost disjoint family $\{A_\alpha\mid\alpha<\kappa^+\}\subseteq[\kappa]^\kappa$ is a \textit{permutable family} if for every bounded $X\subseteq\kappa^+$ there exists a pairwise disjoint family $\{B_\alpha\mid\alpha\in X\}$ such that $A_\alpha=^*B_\alpha$ for all $\alpha\in X$, and for all $\xi<\kappa^+$, $A_\xi\cap\bigcup\{B_\alpha\mid\alpha\in X\}$ is unbounded in $\kappa$ if and only if $\xi\in X$.\footnote{Here, as before, we mean by $=^*$ and $\subseteq^*$ the usual ``almost equal'' or ``almost included'', which means that by changing a bounded subset of $\kappa$ we get the equality or inclusion. We will always be sure to clarify $\kappa$ from the context.}
\end{definition}
\begin{theorem}
Suppose that $\kappa$ is a regular cardinal. Then there exists a permutable family for $\kappa$.
\end{theorem}
\begin{proof}
The proof here is the same as the proof of \autoref{prop:countable sep. family}. Fix a sequence of sets $\tup{T_\alpha\mid\alpha<\kappa^+}$ which is strictly $\subseteq^*$-increasing and define $A_\alpha=T_{\alpha+1}\setminus T_\alpha$. Fix a bounded set $X\subseteq\kappa^+$, then any pairwise disjoint refinement $\{B_\alpha\mid\alpha<\sup X+1\}$ such that $B_\alpha\subseteq T_{\sup X+1}$ witnesses the permutability property.
\end{proof}
\begin{definition}
Fix a permutable family for $\kappa$, $\{A_\alpha\mid\alpha<\kappa^+\}$. We say that $\pi\colon\kappa\to\kappa$ \textit{implements} $\Pi\colon\kappa^+\to\kappa^+$, if for every $\alpha<\kappa^+$, $\pi"A_\alpha=^*A_{\Pi(\alpha)}$.
\end{definition}
\begin{theorem}
Suppose that $\{A_\alpha\mid\alpha<\kappa^+\}$ is a permutable family for $\kappa$. Then for every $\eta<\kappa^+$, and every permutation $\Pi\colon\eta\to\eta$, there is some $\pi\colon\kappa\to\kappa$ such that $\pi$ implements $\Pi$.\qed
\end{theorem}

To simplify later text, we will set some basic notions about what we are going to us permutable families for. Much like the baby step, we will use them to implement bounded permutations of $\kappa^+$ and define a normal filter of subgroups in a similar way to what we did before.

Fix a permutable family for $\kappa$, $\{A_\alpha\mid\alpha<\kappa^+\}$, we say that $\cB$ is a \textit{disjoint approximation} of $\{A_\alpha\mid\alpha\in I\}$ if $\cB$ is a pairwise disjoint family of sets $\{B_\alpha\mid\alpha\in I\}$, such that for every $\alpha\in I$, $A_\alpha=^*B_\alpha$ and $A_\xi\cap\bigcup\cB$ is unbounded in $\kappa$ if and only if $\xi\in I$. As before, we say that $\cB$ is a disjoint approximation, if there is such bounded $I\subseteq\kappa^+$. We say that $\sG$ is \textit{the derived permutation group} from the family if it is the group of all permutations of $\kappa$ which implement a bounded permutation of $\kappa^+$ via the permutable family. Similarly, $\sF$ is \textit{the derived filter of subgroups} if it is the filter of subgroups on $\sG$ generated by \[\{\fix(\cB)\mid\cB\text{ is a disjoint approximation}\},\] where $\fix(\cB)$ is the group $\{\pi\in\sG\mid \pi\restriction\bigcup\cB=\id\}$.

\begin{proposition}
The derived filter is a normal filter of subgroups over the derived permutation group.\qed
\end{proposition}
\subsection{Permutable scales}
Permutable scales are PCF-theoretic objects with properties which mimic permutable families for limit cardinals. Let us fix a limit cardinal $\lambda$ for the rest of this section. We denote by $\SC(\lambda)$ the set $\{\mu^+\mid\omega\leq\mu<\lambda\}$ of successor cardinals below $\lambda$, and by $\jbd(\lambda)$ denote ideal of bounded sets of $\lambda$. Since $\lambda$ is fixed, we write $\jbd$ as a shorthand for $\jbd(\lambda)$.
\begin{definition}
Fix a scale $F=\{f_\eta\mid\eta<\lambda^+\}$ in $\prod\SC(\lambda)/\jbd$. Suppose that $\vec\pi$ is a sequence $\tup{\pi_\theta\mid\theta\in\SC(\lambda)}$ such that $\pi_\theta\colon\theta\to\theta$ is a permutation of $\theta$. We say that $\vec\pi$ \textit{implements} $\pi\colon\lambda^+\to\lambda^+$ if for every large enough $\theta\in\SC(\lambda)$ \[f_{\pi(\eta)}(\theta)=\pi_\theta(f_\eta(\theta)).\]
We say that $F$ is a \textit{permutable scale} if every bounded permutation of $\lambda^+$ can be implemented by some $\vec\pi$. In other words, for every $\eta<\lambda^+$, every permutation of $\eta$ can be implemented.
\end{definition}
The next theorem and its proof are based on the work of the Bristol group.
\begin{theorem}
If $\lambda$ is regular and $2^\lambda=\lambda^+$, or $\lambda$ is singular and $\square^*_\lambda$ holds, then there exists a permutable scale $\{f_\eta\mid\eta<\lambda^+\}$ in $\prod\SC(\lambda)/\jbd$ and $\{\rng f_\eta\mid\eta<\lambda^+\}$ is a permutable family for $\lambda$. 
\end{theorem}
\begin{proof}
If $\lambda$ is regular, we enumerate $\prod\SC(\lambda)$, and by induction define a scale $\{f_\eta\mid\eta<\lambda^+\}$. If $\lambda$ is singular, then by the work of Cummings, Foreman and Magidor in \cite[Theorem~4.1]{CFM:2001}, there exists a scale $\{f_\eta\mid\eta<\lambda^+\}$ in $\prod\SC(\lambda)/\jbd$ with the property that whenever $\eta<\lambda^+$, there is a function $i\colon\eta\to\SC(\lambda)$ such that $\{f_\gamma"[i(\gamma),\lambda)\mid\gamma<\eta\}$ is a family of pairwise disjoint sets.\footnote{In their paper, Cummings, Foreman and Magidor construct such scale on a product of length $\cf(\lambda)$, the same proof works for any product of regular cardinals mutatis mutandis.} Note that in the case of $\lambda$ being regular, this property holds immediately from the regularity of $\lambda$. We shall call such $i\colon\eta\to\SC(\lambda)$ a disjointifying function for $\eta$. This already gives us the wanted property for $\{\rng f_\eta\mid\eta<\lambda^+\}$ to be a permutable family.

We will show that such a scale is indeed a permutable scale. Suppose that $\eta<\lambda^+$ and $\pi\colon\eta\to\eta$ is a permutation of $\eta$, let $i\colon\eta\to\SC(\lambda)$ be a disjointying function for $\eta$. We write $\eta$ as an increasing union of sets $X_\mu$ for $\mu\in\SC(\lambda)$, such that $|X_i|<\mu$. Let $Y_\mu$ be the set $\{\xi<\mu\mid\exists\gamma\in X_\mu: f_\gamma(\mu)=\xi\land\mu>i(\gamma)\}$. By the fact we have disjoint tails, we get that for every $\xi\in Y_\mu$ there is at most one $\gamma\in X_\mu$ witnessing that $\xi\in Y_\mu$. In particular, $|Y_\mu|<\mu$.

Given $\xi\in Y_\mu$, let $\xi^*$ be the unique $\gamma\in X_\mu$ such that $f_\gamma(\mu)=\xi$ and $\mu>i(\gamma)$. Define $\pi_\mu\colon Y_\mu\to Y_\mu$ as follows: For $\xi\in Y_\mu$ such that $\pi(\xi^*)\in X_\mu$ and $i(\pi(\xi^*))<\mu$, define $\pi_\mu(\xi)=f_{\pi(\xi^*)}(\mu)$, this definition is injective on the domain defined so far, so we can extend $\pi_\mu$ to a permutation of $\sup Y_\mu+1$.  

We claim now that $\vec\pi=\tup{\pi_\mu\mid\mu\in\SC(\lambda)}$ implements $\pi$. Suppose that $\gamma<\eta$ and $\overline\gamma=\pi(\gamma)$. For all sufficiently large $\mu$ we have that $\gamma,\overline\gamma\in X_\mu$. If $\mu>i(\gamma),i(\overline\gamma)$, then there are $\xi,\overline\xi\in Y_\mu$ such that $\xi^*=\gamma$ and $\overline\xi^*=\overline\gamma$, and then by the definition of $\pi_\mu$ we have that \[\pi_\mu(f_\gamma(\mu))=\pi_\mu(\xi)=f_{\pi(\gamma)}(\mu)=f_{\overline\gamma}=\overline\xi.\]
So indeed the scale is permutable.
\end{proof}
\begin{remark}
Examining the proof, we can see that we can implement $\pi\colon\lambda^+\to\lambda^+$ using $\vec\pi=\tup{\pi_\mu\mid\mu\in\SC(\lambda)}$ where $\pi_\mu$ is a bounded permutation of $\mu$. This will not be needed in our construction, though.
\end{remark}
Suppose that $\sG_\mu$ is a permutation group of $\mu$, for $\mu\in\SC(\lambda)$, and the full support product $\sG=\prod\sG_\mu$ is such that the $\vec\pi\in\sG$ are enough for witnessing that $F$ is a permutable scale.

We define the derived filter on $\sG$ as follows: first we define $K_{\eta,f}$, for $\eta<\lambda^+$ and $f\in\prod\SC(\lambda)$, to be the following group \[\{\vec\pi\in\sG\mid\iota(\vec\pi)\restriction\eta=\id\text{ and for all }\mu\in\SC(\lambda): \pi_\mu\restriction f(\mu)=\id\},\]
and let $\sF$ be the filter generated by $\{K_{\eta,f}\mid\eta<\lambda^+,f\in\prod\SC(\lambda)\}$.
\begin{proposition}
$\sF$ is a normal filter of subgroups.
\end{proposition}
\begin{proof}
We have $K_{\eta,f}$ and $K_{\eta',f'}$ as defined above, let $\xi=\max\{\eta,\eta'\}$ and let $g$ be $f\vee f'$, the pointwise maximum of $f$ and $f'$, i.e.\ $g(\mu)=\max\{f(\mu),f'(\mu)\}$, then $K_{\xi,g}\subseteq K_{\eta,f}\cap K_{\eta',f'}$. 

To see we have the normality, let $\vec\pi$ be a sequence in $\sG$, and let $\pi$ be $\iota(\vec\pi)$. Since $\pi$ is bounded, there is some $\eta'$ such that $\pi$ only moves ordinals below $\eta'$. Moreover, for every $\mu\in\SC(\lambda)$ we can find $\alpha_\mu$ such that:\begin{enumerate}
\item $\pi_\mu"\alpha_n=\alpha_\mu$, and
\item $\alpha_\mu>f(\mu)$.
\end{enumerate}
Define $g(\mu)=\alpha_\mu$, and $\xi=\max\{\eta,\eta'\}$. Then for $\vec\sigma\in K_{\xi,g}$, it can be readily seen that $\vec\pi^{-1}\vec\sigma\vec\pi\in K_{\xi,g}$.\footnote{Here $\vec\pi^{-1}$ does mean the pointwise inverse.}
\end{proof}
We are now in possession of all the necessary tools for the general construction. Let us finish with one last piece of notation related to $\prod\SC(\lambda)$. If $f\in\prod\SC(\lambda)$, we write $f\da$ to denote the set $\{g\in\prod\SC(\lambda)\mid g(\mu)\leq f(\mu)\text{ for all }\mu\}$.
\section{Giant step: the general construction of the Bristol model}
\subsection{Overview}
The original construction of the Bristol model was using models of the form $L(x)$ for suitable $x$'s. The idea was to add one set after another, and argue from one step to another that the initial segments of the universe become increasingly constant. Then, to argue that the end result is a model which is not $L(x)$ for any set $x$. Using iterations of symmetric extensions we clarify some of the arguments used. We will use the generic sequences to define the next iterand, much like we defined from the $\omega_1$-sequence the second forcing which did not add reals.

We begin by fixing for $\omega$ and for all the successor cardinals permutable families, and permutable scales on $\SC(\lambda)/\jbd(\lambda)$ for every limit cardinal $\lambda$. We have five types of steps: the base step, which is essentially the baby step covered before; the double successor step, which is essentially a retreading of the baby step, replacing $\omega$ by the suitable cardinal; the limit step, where we mainly have to verify that the sequences we have collected so far form a symmetrically generic filter; the limit iterand, where we take the same idea from the other successor case, using the permutable scale to generate the next sequence and overcome the problem of not having sufficient distributivity, while not adding bounded sets; and the successor of the limit step, which is similar to the general successor step, but some fundamental changes must occur due to the nature of the generic sequence of the limit iterand.

In the construction that follows, we separate the cases of $\omega$ (arriving at the $\omega$th iterand, the $\omega$th iterand; and the $\omega+1$th iterand). This will help clarify the arbitrary limit steps, much like our baby step investigation of the first two steps works to clarify the general successor steps.

For every non-limit $\alpha$, we fix a permutable family $\{A^\alpha_\eta\mid\eta<\omega_{\alpha+1}\}$ for $\omega_\alpha$; and for every limit $\alpha$, we fix a permutable scale $F^\alpha=\{f^\alpha_\eta\mid\eta<\omega_{\alpha+1}\}$ in $\prod\SC(\omega_\alpha)/\jbd(\omega_\alpha)$. We will use $\varrho_\alpha$ to denote the generic object added during the $\alpha$th step (so $\varrho_0$ is the Cohen real, $\varrho_1$ is the $\omega_1$-sequence, etc.), $\dot\varrho_\alpha$ Is the canonical $\PP_{\alpha+1}$-name for $\varrho_\alpha$, and $[\dot\varrho_\alpha]_\alpha$ as the canonical $\PP_\alpha$-name of a $\dot\QQ_\alpha$-name for $\varrho_\alpha$.
\subsection{Induction hypothesis}
As is often the case with complicated constructions which have several separated cases, we will have to carry a complicated induction hypothesis. In our case, it is helpful to think about the induction hypothesis as nearly irrelevant for the limit case, and almost entirely about the relationship between $\PP_\alpha$, the iteration so far, and $\dot\QQ_\alpha$, the next iterand.

Assume that we defined $\PP_\alpha$, we assume the following conditions hold:
\begin{enumerate}
\item $\PP_\alpha$ has $\aleph_\alpha$-c.c., and $\tup{\dot\QQ_\beta,\sG_\beta,\sF_\beta\mid\beta<\alpha}$ is a productive iteration.
\item For all $\beta<\alpha$, the following conditions hold:
\begin{enumerate}
\item $\forces_\beta^\IS\tup{\dot\QQ_\beta,\sG_\beta,\sF_\beta}$ is a weakly homogeneous symmetric system.
\item $\PP_\beta\ast\dot\QQ_\beta$ is upwards homogeneous, and $\cG_\beta$ witnesses that.
\item $\forces_\beta\dot\varrho_\beta$ is $\IS_\beta$-generic for $\dot\QQ_\beta$.
\end{enumerate}
\item If $\beta$ is non-limit, then we also assume:
\begin{enumerate}
\item Every condition appearing in $\dot\QQ_\beta$ is a name for a function whose domain is a bounded subset of $\omega_\beta$; $\sG_\beta$ is the derived permutation group on $\omega_\beta$, and $\sF_\beta$ is the derived filter (using the permutable family fixed).
\item $\forces_{\beta+1}\dot\varrho_\beta$ is a function whose domain is $\check\omega_\beta$, and for every $A\in\jbd(\omega_\beta)^L$, $\dot\varrho_\beta\restriction A\in\IS_{\beta+1}$.
\item If $\gaut{\vec\pi}\in\cG_{\beta+1}$, then for all $A\in\jbd(\omega_{\beta+1})^L$, $\gaut{\vec\pi}\dot\varrho_{\beta+1}\restriction A=\gaut{\pi_\beta}\dot\varrho_{\beta+1}\restriction A$.
\item If $\beta=\gamma+1$, then $\forces_\beta^\IS\dot\QQ_\beta$ is $\leq|\dot V_{\omega+\gamma}|$-distributive.
\end{enumerate}
\item If $\beta$ is a limit ordinal, then we also assume:
\begin{enumerate}
\item For every $f\in\prod\SC(\omega_\alpha)$, and $\delta<\alpha$, there is a canonically identified name in $\IS_\alpha$ for $\tup{\dot\varrho_\beta(f(\beta))\mid\beta<\delta}^\bullet$, specifically, we can identify this name from $f$ and $\delta$.
\item $\forces_\beta\dot\varrho_\beta$ is the full support product $\prod_{\gamma<\beta}\dot\varrho_{\gamma+1}$.
\item Let $\dot\varrho_\beta\restriction(E,D)=\tup{\dot\varrho_{\gamma+1}(f(\gamma+1))\mid f\in E,\gamma\in D}$ for $E\subseteq\prod\SC(\omega_\beta)$ and $D\subseteq\beta$. For every $A\in\jbd(\omega_\beta)^L$ and bounded $E\subseteq\prod\SC(\omega_\beta)^L$, $\dot\varrho_\beta\restriction(E,D)\in\IS_\beta$.
\item $\sG_\beta$ is the full support product $\prod_{\gamma<\beta}\sG_{\gamma+1}$, and $\sF_\beta$ is the derived filter.
\item If $\dot x\in\IS_{\beta+1}$, and $\forces_{\beta+1}^\IS\rank(\dot x)<\check\omega+\check\beta$, then there is some $p\in\PP_{\beta+1}$ and $\dot y\in\IS_\beta$ such that $p\forces_{\beta+1}\dot x=\dot y$.
\end{enumerate}
\end{enumerate}
Note that our permutations are all coming from the ground model, as do our filters. We will therefore omit the dots when talking about them, even generically. Namely, we will write $H_\beta$ or $\pi_\beta$ to denote the $\beta$th coordinate of some excellent support $\vec H$ or permutation sequence $\vec\pi$. This will clarify some of the notation and improve the readability. This also means that to verify that the $\alpha$th symmetric system satisfies the productivity clause we only need to check this for $\dot\QQ_\alpha$.\medskip

The productive structure here will be akin to adding Cohen subsets to every successor cardinal, with a full-support product at limit iterand where the permutable scales are utilized to ensure that the coding process can be continued without adding bounded sets.

\subsection{Basis}
The first step is just the baby step. We define $\QQ_0$ to be the Cohen forcing, $\sG_0$ as the group of permutations of $\omega$ which implement a bounded permutation of $\omega_1$ via the permutable family $\{A^0_\eta\mid\eta<\omega_1\}$ and so on. Here $\dot\varrho_0$ is just the canonical name for the Cohen real.

The only minor difference between the baby step part and here is that instead of taking $\power(\omega)^{L[c_\alpha]}$ we take $V_{\omega+1}^{L[c_\alpha]}$, but the translation between the two names is very canonical and of no consequence to the definition. 

\subsection{Successor of non-limits}
Suppose that we have constructed the forcing up to the $\alpha$th step, and $\alpha=\beta+1$ with $\beta$ being a successor or $0$ itself. We have defined $\PP_\alpha=\PP_\beta\ast\dot\QQ_\beta$, as well $\cG_\alpha,\cF_\alpha$ and $\IS_\alpha$ are defined. We shall proceed to define $\dot\varrho_\alpha$ and $\dot\QQ_\alpha,\sG_\alpha$ and $\sF_\alpha$ and prove that they all satisfy the induction hypothesis.

For every $\eta<\omega_\alpha$, we say that $\dot x\in\IS_\alpha$ is almost an $A^\beta_\eta$-name if there is some $A\subseteq\omega_\beta$ such that $A=^*A_\eta^\beta$ and $\dot x$ is a $\PP_\beta\ast(\dot\QQ_\beta\restriction A)$-name. For $\eta<\omega_\alpha$ we define \[\dot R_\eta=\left\{\dot x\in\IS_\alpha\middd\begin{array}{l}
\dot x\text{ is an almost }A^\beta_\eta\text{-name of rank }\leq\omega+\alpha\text{, and}\\
\text{every name which appears in }\dot x\text{ is a }\PP_\beta\text{-name}
\end{array}\right\}^\bullet.\]
Note that this would be exactly the $V_{\omega+\alpha+1}$ of the model obtained by forcing with just $\QQ_\beta\restriction A^\beta_\eta$ over the intermediate model obtained by $\IS_\beta$.

It follows from the induction assumptions that if $\gaut{\vec\pi}\in\cG_\alpha$ and $\pi_\beta$ implements a permutation $\sigma$ of $\omega_\alpha$, then $\gaut{\vec\pi}\dot R_\eta=\dot R_{\sigma(\eta)}$. To see why, first note that if $\dot x$ is an almost $A^\beta_\eta$-name, then $\gaut{\vec\pi}\dot x$ is also an almost $A^\beta_{\sigma(\eta)}$-name; and the result follows. In particular, if $\pi_\beta$ implements $\sigma$ such that $\sigma(\eta)=\eta$, then $\gaut{\vec\pi}\dot R_\eta=\dot R_\eta$. So we get that $\dot R_\eta\in\IS_\alpha$.

Let $\dot\varrho_\alpha=\tup{\dot R_\eta\mid\eta<\omega_\alpha}^\bullet$. By the above, and the assumption on the $\beta$th definition of $\sG_\beta$ and $\sF_\beta$ for every $A\in\jbd(\omega_\alpha)^L$ we have that $\dot\varrho_\alpha\restriction A\in\IS_\alpha$. Now we define $\dot\QQ_\alpha$, ordered by reverse inclusion, \[\dot\QQ_\alpha=\left\{\gaut{\vec\pi}\dot\varrho_\alpha\restriction A\mid\gaut{\vec\pi}\in\cG_\alpha, A\in\jbd(\omega_\alpha)^L\right\}^\bullet.\] 

Next, as we required, we define $\sG_\alpha$ to be the derived permutation group of $\omega_\alpha$; since the names in $\dot\QQ_\alpha$ are all $\bullet$-names for sequences we have a canonical way for defining the names for how such permutation acts on $\dot\QQ_\alpha$. For $\sF_\alpha$ we define it to be the derived filter of subgroups in $\IS_\alpha$. It is immediate that $\tup{\dot\QQ_\alpha,\sG_\alpha,\sF_\alpha}^\bullet$ satisfy the conditions of a productive iteration. Let us verify that the induction hypothesis holds for $\alpha+1$.
\begin{enumerate}
\item The chain condition of $\PP_\alpha\ast\dot\QQ_\alpha$ follows from the fact that in $L[\varrho_0]$ we can embed $\QQ_\alpha$ into the forcing that adds $\omega_\alpha$-Cohen subsets which has $\aleph_{\alpha+1}$-c.c.
\item The fact that $\forces_\alpha^\IS\sG_\alpha\text{ witnesses the homogeneity of }\dot\QQ_\alpha$ is also quite trivial, since given any two conditions $\dot q=\gaut{\vec\pi}\dot\varrho_\alpha\restriction A$ and $\dot q'=\gaut{\vec\tau}\dot\varrho_\alpha\restriction B$ we can find $\sigma$ a bounded permutation of $\omega_{\alpha+1}$ such that if $\dot R_\eta$ lies in the range of both conditions, with $\dot q(\check\gamma)=\dot R_\eta$, then $\dot q'(\sigma(\gamma))=\dot R_\eta$; and otherwise $\sigma$ makes the domains of the two conditions disjoint.

Since every bounded permutation of $\omega_{\alpha+1}$ can be implemented by some permutation of $\omega_\alpha$, we can find some $\pi_\alpha\in\sG_\alpha$ such that $\forces_\alpha\gaut{\pi_\alpha}\dot q$ and $\dot q'$ are compatible.
\item Similar to before, we get upwards homogeneity by noting that we can obtain any bounded permutation of $\omega_\alpha$ via an automorphism in $\sG_\beta$. Therefore we can make any two conditions compatible. Moreover, given an arbitrary $p\in\PP_\alpha$ we can require it to be fixed, simply by noting that we only need to apply a permutation from $\sG_\beta$ so $p\restriction\beta$ can be fixed by default; and $p(\beta)$ has a bounded domain, which we can always assume is part of one of the first set in the disjoint approximation which we use to implement the permutation.
\end{enumerate}
It remains to prove that no sets of rank $<\omega+\alpha$ were added and that $\varrho_\alpha$ is indeed a generic sequence for $\QQ_\alpha$ over the intermediate model. Both of these claims will follow from a generalization of \autoref{lemma:initial segment}.
\begin{lemma}\label{lemma:initial segment general}
Suppose that $\dot D\in\IS_\alpha$ and $p\forces_\alpha^\IS\dot D$ is a dense open subset of $\dot\QQ_\alpha$. Then there is some $\eta<\omega_\alpha$ such that for every $\gaut{\vec\pi}\in\cG_\alpha$ and every $A\in\jbd(\omega_\alpha)$ with $\eta\subseteq A$, if $p\forces_\alpha^\IS\gaut{\vec\pi}\dot\varrho_\alpha\restriction A\in\dot D$ and $\gaut{\vec\tau}$ is such that $p\forces_\alpha^\IS\gaut{\vec\pi}\dot\varrho_\alpha\restriction\eta=\gaut{\vec\tau}\dot\varrho_\alpha\restriction\eta$, then $p\forces_\alpha^\IS\gaut{\vec\tau}\dot\varrho_\alpha\restriction A\in\dot D$.
\end{lemma}
\begin{proof}
The proof is quite similar to the proof of \autoref{lemma:initial segment}. Assume without loss of generality that every name which appears in $\dot D$ also appears in $\dot\QQ_\alpha$. Let $\vec H$ be an excellent support for $\dot D$ and $p$, and $\cB=\{B_\xi\mid\xi<\eta\}$ is a disjoint approximation such that $\fix(\cB)\leq H_\beta$.

By the assumption on $\dot D$ we get that it is enough to consider only automorphisms in $\sG_\beta$. From here the proof is identical to the proof of \autoref{lemma:initial segment}.
\end{proof}
In a similar fashion, these two are corollaries of \autoref{lemma:initial segment general} in the same way we derived the two similar corollaries from \autoref{lemma:initial segment}
\begin{corollary}
$\forces_\alpha\dot\varrho_\alpha$ is $\IS_\alpha$-generic for $\dot\QQ_\alpha$.\qed
\end{corollary}
\begin{corollary}
$\forces^\IS_\alpha\dot\QQ_\alpha$ is $\leq|V_{\omega+\beta}|$-distributive.
\end{corollary}
\begin{proof}
Suppose that $\dot f\in\IS_\alpha$ is a function such that $\dom\dot f$ is $\dot V_{\omega+\beta}$ and $\forces_\alpha\forall x\,\dot f(x)=\dot D_x$ is a dense open subset of $\dot\QQ_\alpha$. Then in $L$, using the chain conditions and the fact that $\dot V_{\omega+\beta}^{\IS_\alpha}$ has cardinality of $\aleph_\beta$ in $L[\varrho_0]$, the proof continues as with \autoref{cor:baby reals}, utilizing the chain condition of $\PP_\alpha$ and the regularity of $\omega_\alpha$, along with the fact that any support for $\dot f$ will invariably support all the $\dot D_x$ uniformly.
\end{proof}
\begin{remark}
It is the same case as with the baby step that we can appeal to absoluteness of distributivity here. In $L[\varrho_0]$ the forcing $\QQ_\alpha$ is $\aleph_\alpha$-distributive, as $\varrho_0$ is a Cohen real and $\QQ_\alpha$ is isomorphic to $\Add(\omega_\alpha,1)^L$. If $\{D_x\mid x\in V_{\omega+\beta}\}$ is a family of dense open set in the intermediate model, then in $L[\varrho_0]$ it has cardinality of $\aleph_\beta$ and therefore its intersection is dense in $\QQ_\alpha$. But the intersection is not dependent on the enumeration, just on the family, so it also lies in our intermediate model.

As before, this would have simplified some of the argument needed in \autoref{lemma:initial segment general} to obtain the genericity of $\varrho_\alpha$, but not sufficiently so to be worth of our effort.\footnote{We will spoil the surprise, and say that this argument also holds for successor of limit ordinals. Now we can skip making the same remark for a third and fourth time.}
\end{remark}
\subsection{To infinity...}
While there is no actual difference between the case where $\alpha=\omega$ and arbitrary limit steps, since we have to address all previous steps in the construction, we invariably need to address previous limit steps. Without the understanding the definitions of the limit stages, this becomes awkward, so in favor of bootstrapping, it is better to separate the cases $\omega$ and $\omega+1$.\medskip

First we need to verify that $\tup{\varrho_n\mid n<\omega}$ is symmetrically generic for $\PP_\omega$.

\begin{proposition}\label{prop:omega-sym-density}
Suppose that $D\subseteq\PP_\omega$ is a symmetrically dense open set, then there is a sequence $\tup{\beta_n\mid n<\omega}$ such that $\tup{\dot\varrho_n\restriction\beta_n\mid n<\omega}\in D$.
\end{proposition}
\begin{proof}
Let $\vec H$ be a support witnessing that $D$ is a symmetrically dense open set. Let $\alpha_n$ be an ordinal such that $H_n$ contains $\fix(\cB)$ which is a disjoint approximation of $\{A^n_\eta\mid\eta<\alpha_n\}$.

Let $p$ be the condition such that $p(n)=\dot\varrho_n\restriction\alpha_n$. Let $r_0\leq_\omega p$ such that $r_0\in D$. We will construct a finite sequence of length $m+1$, of conditions $r_k$ for $k\leq m$, such that $r_k\leq_\omega p$, $r_k\in D$ and $\supp(r_k)=\supp(r_0)$, with $r_m$ satisfying \[r_m(n)=\dot\varrho_n\restriction\beta_n\text{ for some }\beta_n,\ \text{ for all }n.\]
(It should be remarked that it is clear that $m\leq|\supp(r_0)|$, since these are the only coordinates it could be incompatible with $\dot\varrho_n$'s to begin with.)

Suppose that $r_k$ was defined. If $r_k$ is not of the wanted form, let $i$ be the least such that $r_k(i)$ is incompatible with $\dot\varrho_i$. By homogeneity, let $\pi_i$ be an automorphism in $\sG_i$ such that $\forces_i^\IS\pi_i r_k(i)$ is compatible with $\dot\varrho_i$. Moreover, we can find one such that $\dot\varrho_i\restriction\alpha_i$ is fixed, by the choice of $\alpha_i$. Let $r_{k+1}=\gaut{\pi_i}r_k$, then by the choice of $\pi_i$, and the fact $r_k\in D$ it follows that $r_{k+1}\in D$ as well. As $\supp(r_0)$ is finite, this process has to terminate in some $m+1$ steps as wanted. Now $r_m\in D$ and it is compatible with the sequence of generics, so it has an extension as wanted inside $D$ by the fact $D$ is open.
\end{proof}

There is not much else to verify in this case, since $\tup{\dot\varrho_n(f(n))\mid n<k}^\bullet\in\IS_\omega$ for every $f\in\prod\SC(\omega_\omega)$ and $k<\omega$, for obvious reasons: every such name is finite.

\subsection{\ldots and beyond! The \texorpdfstring{$\omega$}{w}th iterand}
Let $f\in\prod\SC(\omega_\omega)$, define $\dot\varrho_{\omega,f}$ to be $\tup{\dot\varrho_{n+1}(f(n))\mid n<\omega}^\bullet$, and let $\dot\varrho_\omega$ be $\tup{\dot\varrho_{\omega,f}\mid f\in\prod\SC(\omega_\omega)}^\bullet$. If $E\subseteq\prod\SC(\omega_\omega)$ and $A\subseteq\omega$, we will write $\dot\varrho_\omega\restriction(E,A)=\tup{\dot\varrho_{\omega,f}\restriction A\mid f\in E}^\bullet$.

For a fixed $f\in\prod\SC(\omega_\omega)$, we define $\dot\QQ_{\omega,f}$ as the following forcing, \[\left\{\gaut{\vec\pi}\dot\varrho_{\omega,f}\restriction A\middd\gaut{\vec\pi}\in\cG_\omega, A\in\jbd(\omega)\right\}^\bullet,\]
with the order being reverse inclusion. The idea is that $\dot\QQ_{\omega,f}$ is the forcing which will add back only $\dot\varrho_{\omega,f}$. However, we are interested in the entire ``product'' of the generics, not just one section along the product.  We define $\dot\QQ_\omega$ to be as follows,
\[\left\{\gaut{\vec\pi}\dot\varrho_\omega\restriction(E,A)\middd\begin{array}{l}
\gaut{\vec\pi}\in\cG_\omega,\\
E\subseteq\prod\SC(\omega_{\omega})\text{ bounded, and}\\
A\in\jbd(\omega_\omega)
\end{array}\right\}^\bullet.\]
If $\dot q=\gaut{\vec\pi}\dot\varrho_\omega\restriction(E,A)$ is a condition in $\dot\QQ_\omega$, $\dot q(f)$ denotes the $f$th sequence inside $\dot q$. Namely, $\dot q(f)=\tup{\gaut{\vec\pi}\dot\varrho_{n+1}(f(n))\mid n\in A}$. We also write $\dot q(f,n)$ to denote the $n$th coordinate of $\dot q(f)$.

We order $\dot\QQ_\omega$ as follows: $\dot q\leq_{\dot\QQ_\omega}\dot q'$ if for every $f\in\prod\SC(\omega_\omega)$, $\dot q(f)\leq_{\dot\QQ_{\omega,f}}\dot q'(f)$.

\begin{proposition}\label{prop:omega-up-hom}
If $p\in\PP_\omega$ and $\dot q$ is a condition in $\dot\QQ_\omega$, then there is some automorphism $\gaut{\vec\pi}\in\cG_\omega$ such that $\gaut{\vec\pi} p = p$ and $\gaut{\vec\pi}\dot q$ is compatible with $\dot\varrho_\omega$.
\end{proposition}
\begin{proof}
By the definition of $\dot\QQ_\omega$ we know there is some $\vec\sigma$, $E$ and $A$ in the ground model such that $\dot q=\gaut{\vec\sigma}\dot\varrho_\omega\restriction(E,A)$.

We define $\vec\pi$ by induction. We only need to define finitely many coordinates, since $A$ is a finite subset of $\omega$. Let $\pi_0$ be an automorphism in $\cG_1$ such that $\pi_0 p_0=p_0$ and $\gaut{\pi_0}\dot q$ is compatible with $\dot\varrho_\omega$ on the first coordinate. Such $\pi_0$ can be found by the upwards homogeneity of $\PP_1\ast\dot\QQ_1$. 

We proceed by induction in a similar fashion. If $\vec\pi\restriction k$ was defined, we can find some $\pi_k\in\sG_k$ such that $\pi_k p_k=p_k$ and $\gaut{\pi_k}\gaut{\vec\pi\restriction k}\dot q$ agrees with $\dot\varrho_\omega$ on the coordinate $k$, and such $\pi_k$ can be found by the upwards homogeneity assumptions for $k$. Finally, note that $\gaut{\pi_k}\gaut{\vec\pi\restriction k}=\gaut{\vec\pi\restriction k}\gaut{\gaut{\vec\pi\restriction k^{-1}}\pi_k}=\gaut{\vec\pi\restriction k}\gaut{\pi_k}=\gaut{\vec\pi\restriction k+1}$. 
\end{proof}
\begin{corollary}$\PP_\omega\ast\dot\QQ_\omega$ is upwards homogeneous.\qed
\end{corollary}

The automorphism group, as the induction hypothesis requires is the group of all sequences $\vec\pi\in\prod_{n<\omega}\sG_{n+1}$ such that $\vec\pi$ implements a bounded permutation of $\omega_{\omega+1}$ using the permutable scale $F^\omega$. Such $\vec\pi$ acts naturally on a condition in $\dot\QQ_{\omega,f}$ by a pointwise action for every $f\in\prod\SC(\omega_\omega)$, and this action extends to $\dot\QQ_\omega$. As $\sG_{n+1}$ witnesses the homogeneity of $\dot\QQ_{n+1}$, $\sG_\omega$ witnesses the homogeneity of $\dot\QQ_\omega$.
We will use $\vec\pi\dot q$ to denote the action of $\vec\pi$ on $\dot q$. This is to discern the action from the action of $\gaut{\vec\pi}$, as the two groups are very different: the composition in $\cG_\omega$ is a successive application of automorphisms, whereas in $\sG_\omega$ it is just pointwise composition. The cardinality of $\sG_\omega$ is of course bounded by the cardinality of the product, which by the assumption of $\GCH$ is exactly $\aleph_{\omega+1}$.

As the induction hypothesis revealed, $\sF_\omega$ is the derived filter on $\sG_\omega$.

\begin{lemma}\label{lemma:omega-genericity}
Suppose that $p\forces^\IS_\omega\dot D\subseteq\dot\QQ_\omega$ is dense open. Then there are ground model sets $E\subseteq\prod\SC(\omega_\omega)$ and $A\in\jbd(\omega)$ such that $q\forces^\IS_\omega\dot\varrho_\omega\restriction(E,A)\in\dot D$ for some $q\leq_\omega p$.
\end{lemma}
\begin{proof}
Let $\vec H$ be an excellent support for $\dot D$ and $p$. For every $n$, let $\cB_n$ be a disjoint approximation such that $\fix(\cB_n)\leq H_n$. Let $B_n$ be $\bigcup\cB_n$ for $n\in C(\vec H)$ and $\{0\}$ otherwise. Let $E'=\prod_{n<\omega}B_{n+1}$ and $A'=C(\vec H)$, and let $q'\leq_\omega p$, $\gaut{\vec\pi}$ and $E,A$ such that $q'\forces_\omega^\IS\gaut{\vec\pi}\dot\varrho_\omega\restriction(E,A)\in\dot D$ and $\gaut{\vec\pi}\dot\varrho_\omega\restriction(E,A)\leq_{\dot\QQ_\omega}\dot\varrho_\omega\restriction(E',A')$.

By the assumption that $q'\forces^\IS_\omega\gaut{\vec\pi}\dot\varrho_\omega\restriction(E,A)\leq_{\dot\QQ_\omega}\dot\varrho_\omega\restriction(E',A')$, we know that for every $f\in E'$ and $n\in A'$, $\gaut{\pi_n}(\dot\varrho_{n+1}(f(n)))=\dot\varrho_{n+1}(f(n))$. By the very choice of $B_n$'s we can find a permutation which implements $\pi_n^{-1}$ in $H_n$. Therefore we can find some $\vec\sigma$ such that $\gaut{\vec\sigma}\gaut{\vec\pi}\dot\varrho_\omega\restriction(E,A)=\dot\varrho_\omega\restriction(E,A)$ and $\gaut{\vec\sigma}p=p\forces_\omega^\IS\gaut{\vec\sigma}\dot D=\dot D$. Taking $q=\gaut{\vec\sigma}q'$ satisfies by the Symmetry Lemma that $q\forces^\IS_\omega\dot\varrho_\omega\restriction(E,A)\in\dot D$.
\end{proof}
\begin{corollary}
$\varrho_\omega$ is generic for $\QQ_\omega$ over the intermediate model of $\IS_\omega$.
\end{corollary}
\begin{proof}
Assume otherwise, then there is some generic filter $G$ for $\PP_\omega$ witnessing otherwise. Then we can find $\dot D\in\IS_\omega$ and $p\in G$ such that for every $(E,A)$ we have $p\forces_\omega\dot\varrho_\omega\restriction(E,A)\notin\dot D$. Of course, this is a contradiction to the previous lemma, so the wanted genericity follows.
\end{proof}
Next we need to ensure that there are no bounded sets added. At the double successor steps, we had the luxury of the forcing satisfying a very strong distributivity condition. In this case, however, this is impossible, as we aim to add each $\varrho_{\omega,f}$, which is an $\omega$-sequence. Here is where the properties of the symmetric system will help us. It will be easier to finish this part of the proof working inside the intermediate model of $\IS_\omega$. So $\HS$ now means the hereditarily $\QQ_\omega$-symmetric names in our intermediate model, and so on.
\begin{definition}
Let $\dot x$ be a $\QQ_\omega$-name. We say that $\dot x$ is bounded by $f\in\prod\SC(\omega_\omega)$ if whenever $\gaut{\vec\pi}\varrho_\omega\restriction(E,A)\forces_{\QQ_\omega}\dot y\in\dot x$, then we can restrict $E$ to $E\cap f\da$; similarly $\dot x$ is bounded by $n$ if we can restrict $A$ to $A\cap n$. We will write $q\restriction f$ and $q\restriction n$ to denote these restriction for a condition $q\in\QQ_\omega$.
\end{definition}
\begin{remark}
To the readers familiar with standard symmetric extensions arguments, e.g., Cohen's first model, the above is similar to how forcing a symmetric statement can be restricted to the support of that statement. Here we are going to utilize the whole iteration to obtain similar results. We will only consider the above definition when $\dot x$ is a name for a subset of the ground model ($\IS_\omega$ in this case), such that the names appearing inside are canonical names for ground model elements.
\end{remark}
\begin{lemma}\label{lemma:omega-sym-dist}
Suppose that $\dot x$ is a name for a subset of the ground model (i.e.\ $\IS_\omega$), such that every $\dot y$ appearing in $\dot x$ is a canonical name for a ground model element. 
\begin{enumerate}
\item If $\dot x\in\HS$, then there is some $f$ such that $\dot x$ is bounded by $f$, 
\item and if $\forces_{\QQ_\omega}\rank(\dot x)<\check\omega+\check\omega$, then $\dot x$ is bounded by some $n<\omega$.
\end{enumerate}
\end{lemma}
We get the following easy corollary, which now ensures that there are no sets of rank $<\omega+\omega$ added after forcing with the symmetric system $\tup{\dot\QQ_\omega,\sG_\omega,\sF_\omega}$.
\begin{corollary}
If $\dot x\in\IS_{\omega+1}$, and $p\forces_{\omega+1}^\IS\rank(\dot x)<\check\omega+\check\omega$, then there is some $q\leq_{\omega+1}p$ and some $\dot y\in\IS_\omega$ such that $q\forces_{\omega+1}^\IS\dot x=\dot y$.
\end{corollary}
\begin{proof}
We prove by induction on $n$ that if $\dot x$ is forced to have rank $\omega+n$, then the conclusion holds. Assume the induction hypothesis for $k<n$. We may assume that all the names which appear in $\dot x$ are names in $\IS_\omega$ (in fact in $\IS_{n+1}$ by the distributivity of previous steps), otherwise whenever $\tup{p',\dot y}$ appear in $\dot x$ with $p'$ compatible with $p$, we can replace the pair by extending $p'$ to a maximal antichain (or a dense open set) of conditions satisfying the conclusion of the corollary for $\dot y$, with suitable names from $\IS_\omega$. By the Symmetry Lemma we get that this refined name is respected exactly by the same automorphisms as $\dot x$, and it is forced to be equal to $\dot x$.

By (1) we get that $\dot x$ is bounded by $f$ and $n$, for some suitable $f$ and $n$, on its $\omega$-coordinate. This means that extending any condition in $\dot\QQ_\omega$ to $(f\da,n)$ will invariably turn $\dot x$ into a name in $\IS_\omega$.
\end{proof}
\begin{proof}[Proof of \autoref{lemma:omega-sym-dist}]
We begin by proving (1):

Suppose that $\dot x\in\HS$, let $K_{\eta,f}$ be such that $K_{\eta,f}\leq\sym(\dot x)$ and let $f_\eta$ denote $f^\omega_\eta\in F^\omega$, the $\omega$th permutable scale. We may assume without loss of generality that $f_\eta(n)\leq f(n)$ for all $n<\omega$, otherwise we can shrink $K_{\eta,f}$ to $K_{\eta,f_\eta\vee f}$. We claim that $f$ bounds $\dot x$. 

If $q\forces_{\QQ_\omega}\dot y\in\dot x$ we may assume without loss of generality that $\tup{q,\dot y}\in\dot x$ (recall that from $q$'s perspective $\dot y$ is in fact $\check a$ for some $a$), first note that by the very definition of conditions in $\QQ_\omega$, if $q=\gaut{\vec\pi}\varrho_\omega\restriction(E,A)$ and $f,g\in E$ then for all $n$, if $g(n)=f(n)$, then we have that $q(f,n)=q(g,n)$. Let $q'$ be an extension of $\gaut{\vec\pi}\varrho_\omega\restriction(f\da,A)$.

Then there are some $E',A'$ and $\vec\sigma$ such that $q'=\gaut{\vec\sigma}\varrho_\omega\restriction(E',A')$. By the pointwise homogeneity of each $\QQ_n$, we can find a sequence $\vec\tau$ which is cofinitely the identity, and $\vec\tau(q')$ is compatible with $q$. Moreover, by the fact that $\vec\tau$ is cofinitely the identity, we get that $\iota(\vec\tau)=\id$ and we may choose each $\tau_n$ to only move coordinates above $f(n)$, as all the coordinates below $f(n)$ are agreed upon by both $q$ and $q'$. In other words, $\vec\tau\in K_{\eta,f}$ as wanted.\medskip

For the proof of (2), we prove by induction on the $\QQ_\omega$-name rank of $\dot x$. Of course, it is enough to prove the claim for the names of rank $<\omega+\omega$. Suppose that $\dot x$ has rank $\omega+n$ and the claim is true for all the names which appear in $\dot x$. By refining the conditions appearing inside $\dot x$, we may even assume that every name appearing in $\dot x$ is a canonical name for a ground model element (looking at this statement from $L$, this means that all the names are already from $\IS_\omega$, and in fact $\IS_n$, by the distributivity argument for successor steps). 

Let $[\dot x]$ be the $\PP_\omega$-name in $\IS_\omega$ for $\dot x$, and let $\vec H$ be an excellent support for $[\dot x]$ with $k=\max C(\vec H)$, without loss of generality $k\geq n$. By the assumption on the rank, we may assume that whenever $\dot y$ appears in $\dot x$, then $[\dot y]$ is a canonically generated from a name in $\IS_n$. Let $\tup{q,\dot y}\in\dot x$, if we cannot bound the condition $q$ by $k$, then there is some $p\in\PP_\omega$ forcing that. Let $q'\leq_{\QQ_\omega}q\restriction k$, such that $p\forces_\omega^\IS\dot q'\forces_{\dot\QQ_\omega}[\dot y]\notin[\dot x]$.

Using the homogeneity of $\PP_\omega$ and the upwards homogeneity of $\PP_\omega\ast\dot\QQ_\omega$, we can find some $\vec\pi$ such that $\vec\pi\restriction k+1=\id$, $\gaut{\vec\pi}p=p$ and $p$ forces that $\gaut{\vec\pi}\dot q'$ is compatible with $\dot q$. But by the choice of $\vec\pi$ we get that $\gaut{\vec\pi}[\dot x]=[\dot x]$ and $\gaut{\vec\pi}[\dot y]=[\dot y]$. Therefore $p\forces_\omega^\IS\gaut{\vec\pi}q'\forces_{\dot\QQ_\omega}[\dot y]\notin[\dot x]$ which is a contradiction to the assumption on $q$.
\end{proof}
We finish with the following proposition which will be necessary for the next step, as well as for the induction hypothesis for the general limit case.
\begin{proposition}\label{prop:omega-coherence}
Suppose that $f\in\prod\SC(\omega_\omega)$, then $\dot\varrho_{\omega,f}\in\IS_{\omega+1}$.
\end{proposition}
\begin{proof}
Let $f+1$ denote the function such that $(f+1)(n)=f(n)+1$, then $K_{0,f+1}$ must preserve $\dot\varrho_{\omega,f}$.
\end{proof}
An important remark to make about the claim above is that the name is in fact in $\IS_{\omega+1}$ and not in $\IS_\omega$, since objects in $\IS_\omega$ are somewhat invariant upto some finite generic information, where $\varrho_{\omega,f}$ clearly does not satisfy this property. However, $\QQ_\omega$ allows us to condense this information into a name which we can represent canonically as a $\PP_{\omega+1}$-name. It is important to note that if $\gaut{\vec\pi}\in\cG_\omega$, then $\gaut{\vec\pi}\dot\varrho_{\omega,f}$ is what we would expect it to be, simply by the nature of the fact that $\dot\varrho_{\omega,f}$ is a $\dot\QQ_{\omega,f}$-name, and the way that $\gaut{\vec\pi}$ acts on $\dot\QQ_{\omega,f}$.
\subsection{The first transfinite successor}
Next, we arrive to the successor of $\omega$. This part is not covered by the general successor construction, although the general idea is in fact the same, since $\varrho_\omega$ has a different structure compared to $\varrho_{\alpha+1}$. So we need to separate this case as well. For readability we will write $f_\eta$ for $f^\omega_\eta$, the $\eta$th function in $F^\omega$.

We say that $\dot x$ is an almost $f_\eta$-name, if there is some $f\in\prod\SC(\omega_\omega)$ such that $f=^*f_\eta$ and $\dot x$ is a $\PP_\alpha\ast\dot\QQ_{\omega,f}$-name. As with the double-successor construction, we define $\dot R_\eta$ as follows,
\[\dot R_\eta=\left\{\dot x\in\IS_{\alpha+1}\middd\begin{array}{l}
\dot x\text{ is an almost }f_\eta\text{-name of rank }\leq\omega+\omega,\text{ and}\\
\text{every name that appears in }\dot x\text{ is a }\PP_\omega\text{-name}
\end{array}\right\}^\bullet.\]
The idea is that this is $V_{\omega+\omega+1}$ of the model obtained by forcing with only $\QQ_{\omega,f_\eta}$. Now we define $\dot\varrho_{\omega+1}=\tup{\dot R_\eta\mid\eta<\omega_{\omega+1}}^\bullet$. 
\begin{proposition}\label{prop:omega-successor-conditions}
Suppose that $A\subseteq\omega_{\omega+1}$ is bounded, then $\dot\varrho_{\omega+1}\restriction A\in\IS_{\omega+1}$.
\end{proposition}
\begin{proof}
By the very definition of $\dot\varrho_{\omega+1}\restriction A$, it is easy to see that it is stable under every $\gaut{\vec\pi}\in\cG_\omega$. So it is enough to find a supporting group in $\sF_\omega$. Let $\eta>\sup A$, then we claim that $K_{\eta,f_\eta}$ does the work.

Suppose that $\vec\pi\in K_{\eta,f_\eta}$, then $\vec\pi$ implements the identity function on $A$. In particular, for every $\alpha\in A$, $\vec\pi\,"\,\dot\QQ_{\omega,f_\alpha}$ equals to some $\dot\QQ_{\omega,f}$ for some $f=^* f_\alpha$. This means that being an almost $f_\alpha$-name is preserved under $\vec\pi$. Therefore the name $\dot\varrho_{\omega+1}\restriction A$ is not moved at all.
\end{proof}

And as before we define $\dot\QQ_{\omega+1}$ to be $\{\gaut{\vec\pi}\dot\varrho_{\omega+1}\restriction A\mid \gaut{\vec\pi}\in\cG_{\omega+1}, A\in\jbd(\omega_{\omega+1})^L\}^\bullet$. The proof of the above proposition shows that indeed $\dot\QQ_{\omega+1}$ is respected by all the permutations, so indeed it is a valid candidate for the next step.
\begin{remark}
Note that when applying automorphisms from $\cG_{\omega+1}$ we only need to care about those coming from $\sG_\omega$. In light of this, we can keep the somewhat confusing notation of $\vec\pi$. We will, however, use $\vec\pi^\omega$ to denote the sequence in $\sG_\omega$ and $\pi^\omega_n$ to denote its $n$th coordinate.
\end{remark}

The only part which is significantly different from the general successor construction is in the upwards homogeneity, so we will conclude our first transfinite exploration by proving that $\PP_{\omega+1}\ast\dot\QQ_{\omega+1}$ is indeed upwards homogeneous.
\begin{proposition}\label{prop:omega+1-up-hom}
$\PP_{\omega+1}\ast\dot\QQ_{\omega+1}$ is upwards homogeneous.
\end{proposition}
\begin{proof}
It is enough to show that if $\tup{p,\dot q}\in\PP_{\omega+2}$, then there is some $\vec\pi\in\cG_{\omega+1}$ such that $\gaut{\vec\pi}\dot q=\dot\varrho_{\omega+1}\restriction A$ for some $A$, and $\gaut{\vec\pi}p=p$.

By definition there is some $\vec\sigma$ such that $\dot q=\gaut{\vec\sigma}\dot\varrho_{\omega+1}\restriction A$. We may ignore all the finite coordinates of $\vec\sigma$ and focus only on $\vec\sigma^\omega$. Let $k$ be large enough such that $p_\omega$ is bounded by $k$, then any $\vec\pi^\omega$ such that $\vec\pi^\omega\restriction k=\id$ is guaranteed not to move $p$ at all. Let $\pi^\omega_n=(\sigma^\omega_n)^{-1}$ for $n\geq k$, then we claim that $\gaut{\vec\pi^\omega}\dot q=\dot\varrho_{\omega+1}\restriction A$. But it is obvious that for every $\alpha\in A$, $\gaut{\vec\pi^\omega}\dot q(\alpha)=\dot\varrho_{\omega+1}(\alpha)$. And the conclusion follows.
\end{proof}
The rest of the construction, including the definitions of $\sG_{\omega+1}$ and $\sF_{\omega+1}$ are the same as the usual successor step, and there is no need to modify the constructions.
\subsection{Other limit steps}
The rest of the section will be devoted to the general limit-related steps. These are similar to the $\omega$-related steps,\footnote{We apologize in advance to the reader: in some of these cases the proofs are quite the same argument, perhaps with a small change. These changes will be noted if they are not sufficiently clear. However, in some of the proofs it might seem that there is a proof using the inductive construction, where a ``direct proof'' can be given as in the $\omega$-related steps. Of course these statements can be proved in such way, but the ``direct proof'' also works. We have arrived to a no-win scenario in wasting the reader's time: retread the same proof as before, come up with it yourself, or figure out why the direct proof works rather than a proof using the construction up to $\alpha$. To add insult to injury, this footnote would have taken a few minutes to read as well. Sorry.} although in several points we will need to address previous limit cases, which means that we use the bootstrapping case of $\omega$, and not-yet-proved statements which will be formulated and proved in the rest of this section. We mainly have to verify that the sequence $\tup{\varrho_\beta\mid\beta<\alpha}$ is symmetrically generic for $\PP_\alpha$ when $\alpha$ is a limit. 

\begin{proposition}
Suppose that $D\subseteq\PP_\alpha$ is a symmetrically dense set. Then there is some sequence $\tup{\eta_\beta\mid\beta<\alpha}$ such that $\tup{\varrho_\beta\restriction\eta_\beta\mid\beta<\alpha}^\bullet\in D$, where $\eta_\beta$ is an ordinal for non-limit $\beta$, and $(E_\beta,A_\beta)$ such that $E_\beta\subseteq\prod\SC(\omega_\beta)$ and $A_\beta\subseteq\beta$ are bounded sets.
\end{proposition}
\begin{proof}
Let $\vec H$ be a support for $D$, and let $\zeta_\beta$ be such that $H_\beta$ contains $\fix(\cB_\beta)$ which is a disjoint approximation for $\{A^\beta_\eta\mid\eta<\zeta_\beta\}$, and $\zeta_\beta=(E_\beta,A_\beta)$ such that $K_{\xi^\beta,f^\beta}$ is a subgroup of $H_\beta$ with $f^\beta$ an upper bound of $E_\beta$ and $A_\beta=\max C(\vec H\restriction\beta)+1$. Namely, $E_\beta$ is the set of functions dominated by $f^\beta$, and $A_\beta$ is the initial segment where $\vec H$ still has the possibility to be nontrivial. If $\beta\notin C(\vec H)$, we define $\zeta_\beta=0$ or $(\varnothing,\varnothing)$.

The rest of the proof goes almost the same as the proof of \autoref{prop:omega-sym-density}, we start with an extension of $p=\tup{\dot\varrho_\beta\restriction\zeta_\beta\mid\beta<\alpha}^\bullet$, and extend it to $r_0\in D$. Then we proceed by induction to ``correct'' the incompatible coordinates of $r_0$ one by one. The only difference is that now we might have the  case where we need to correct a condition in $D$ at a limit coordinate, $\beta$. The ``nature'' of the automorphism there is slightly different, but the essence remains: we can find an automorphism which will not move coordinates inside $A_\beta$.
\end{proof}

The chain condition of $\PP_\alpha$ is indeed $\aleph_\alpha$-c.c., as a finite support iteration of forcing posets with even smaller chain conditions. So it remains to verify the condition on $\tup{\dot\varrho_\beta(f(\beta))\mid\beta<\alpha}^\bullet\in\IS_\alpha$.
\begin{proposition}
Let $f\in\prod\SC(\omega_\alpha)$ and $\delta<\alpha$, then $\tup{\dot\varrho_\beta(f(\beta))\mid\beta<\delta}^\bullet$ has a name in $\IS_\alpha$ which is definable from $\delta$ and $f$.
\end{proposition}
\begin{proof}
If $\alpha=\alpha'+\omega$ for some $\alpha'<\alpha$, then either $\delta<\alpha'$ in which case the induction hypothesis for $\alpha'$ ensures the existence of such name; or $\delta=\alpha'+n$ for some $n<\omega$, in which case by \autoref{prop:coherence} (and \autoref{prop:omega-coherence} as a bootstrapping case) we get can extend the name for $\dot\varrho_{\alpha',f\restriction\alpha'}$ by adding the finite part $\tup{\dot\varrho_{\alpha'+n+1}(f(\alpha'+n+1))\mid\alpha'+n<\delta}^\bullet$.

If $\alpha$ is a limit of limit ordinals, then the induction hypothesis suffices to immediately get this result, as any restriction is bounded below some previous limit ordinal.
\end{proof}
\subsection{Limit iterands}
For $f\in\prod\SC(\omega_\alpha)$, let $\dot\varrho_{\alpha,f}$ denote $\tup{\dot\varrho_{\beta+1}(f(\beta+1))\mid\beta<\alpha}^\bullet$. We define $\dot\varrho_\alpha$ to be $\tup{\dot\varrho_{\alpha,f}\mid f\in\prod\SC(\omega_\alpha)}^\bullet$, similar to how $\dot\varrho_\omega$ was defined. For $E\subseteq\prod\SC(\omega_\alpha)$ and $A\subseteq\alpha$, we write $\dot\varrho_\alpha\restriction(E,A)$ to denote $\tup{\dot\varrho_{\alpha,f}\restriction A\mid f\in E}^\bullet$.

For $A\in\jbd(\omega_\alpha)$ we say that $\delta$ is the \textit{condensation point of $A$} if $\delta$ is a limit ordinal and $\delta+\omega=\alpha$, or if $\delta$ is the least limit ordinal such that $A\subseteq\delta$ in the case that $\alpha$ is a limit of limit ordinals.

\begin{remark}
It will be important later on that we insist that $\dot\varrho_{\alpha,f}\restriction A$ and $\dot\varrho_\alpha\restriction(E,A)$ are composed from the names which we can identify as these restrictions in $\IS_\alpha$. More importantly, we will require them to be the names which were obtained in $\IS_{\delta+1}$, where $\delta$ is the condensation point of $A$, with a finite addition if necessary.
\end{remark}

Define $\dot\QQ_{\alpha,f}$ to be the forcing $\{\gaut{\vec\pi}\dot\varrho_{\alpha,f}\restriction A\mid A\in\jbd(\omega_\alpha)\}^\bullet$, ordered by reverse inclusion. And $\dot\QQ_\alpha$ to be the following forcing: \[\left\{\gaut{\vec\pi}\dot\varrho_\alpha\restriction(E,A)\middd\begin{array}{l}\gaut{\vec\pi}\in\cG_\alpha,\\ E\subseteq\prod\SC(\omega_\alpha)\text{ bounded, and}\\ A\in\jbd(\alpha)
\end{array}\right\}^\bullet.\]
We use the same notation as in the case $\alpha=\omega$ for $\dot q(f)$ and $\dot q(f,\xi)$. Now define $\dot q\leq_{\dot\QQ_\alpha}\dot q'$ if and only if for all $f\in\prod\SC(\omega_\alpha)$, $\dot q(f)\leq_{\dot\QQ_{\alpha,f}}\dot q'(f)$.
\begin{proposition}
If $\tup{p,\dot q}\in\PP_\alpha\ast\dot\QQ_\alpha$, then there is some $\gaut{\vec\pi}\in\cG_\alpha$ such that $\gaut{\vec\pi}p=p$ and $\gaut{\vec\pi}\dot q$ is compatible with $\dot\varrho_\alpha$.
\end{proposition}
\begin{proof}
The proof is the same proof as \autoref{prop:omega-up-hom}, with one change. Here we start the induction at $\delta$, the condensation point of our two conditions in $\dot\QQ_\alpha$, and utilize the definition of $\sG_\delta$ to find $\pi_\delta$ as needed.
\end{proof}
\begin{corollary}
$\PP_\alpha\ast\dot\QQ_\alpha$ is upwards homogeneous.\qed
\end{corollary}

The following proposition is a direct analog of its $\omega$th counterpart.
\begin{proposition}
$\forces_\alpha^\IS\sG_\alpha$ witnesses the homogeneity of $\dot\QQ_\alpha$.\qed
\end{proposition}
\begin{lemma}
Suppose that $\dot D\in\IS_\alpha$ and $p\forces^\IS_\alpha\dot D\subseteq\dot\QQ_\alpha$ is dense. Then there is some $q\leq_\alpha p$ and $(E,A)$ such that $q\forces_\alpha^\IS\dot\varrho_\alpha\restriction(E,A)\in\dot D$.
\end{lemma}
\begin{proof}
Let $\vec H$ be a support witnessing that $\dot D\in\IS_\alpha$ and fixing $p$. Let $\delta$ be the condensation point of $C(\vec H)$, and let $\eta$ be the maximum of $\delta$ and $\max C(\vec H)+1$. Let $\eta_\delta$ and $f_\delta$ such that $H_\delta=K_{\eta_\delta,f_\delta}$. For a successor $\beta<\alpha$ we define the set $E_\beta$ as follows:
\begin{enumerate}
\item If $\beta<\delta$, then $E_\beta=f_\delta(\beta)+1$.
\item If $\eta>\beta>\delta$, then $E_\beta=\eta_\beta$ is some non-zero ordinal such that for some $\cB_\beta$, $\fix(\cB_\beta)\leq H_\beta$, and $\cB_\beta$ is a disjoint approximation of $\{A^\beta_\gamma\mid\gamma<\eta_\beta\}$.
\item If $\beta\geq\eta$, then $E_\beta=\{0\}$.
\end{enumerate}

We get that $E'=\prod_{\beta<\alpha}E_\beta$ is a bounded set (bounded by $f_\delta$ extended by $\eta_\beta$ or $0$ where needed). Therefore $\dot\varrho_\alpha\restriction(E',\eta)$ is a condition in $\dot\QQ_\alpha$. Let $q'\leq_\alpha p$ and let $\gaut{\vec\pi}\dot\varrho_\alpha\restriction(E,A)$ be a condition such that \[q'\forces^\IS_\alpha\gaut{\vec\pi}\dot\varrho_\alpha\restriction(E,A)\in\dot D\text{ and }\gaut{\vec\pi}\dot\varrho_\alpha\restriction(E,A)\leq_{\dot\QQ_\alpha}\dot\varrho_\alpha\restriction(E',\eta).\]
The proof is now similar to the proof of \autoref{lemma:omega-genericity}.
\end{proof}
\begin{corollary}
$\varrho_\alpha$ is generic for $\QQ_\alpha$ over the intermediate model of $\IS_\alpha$.\qed	
\end{corollary}

Again, in parallel to the $\omega$th iterand, we want to prove that no sets of rank below $\omega+\alpha$ were added over $\IS_\alpha$. The proofs are painfully similar. We work in the interpretation of $\IS_\alpha$ as a model of $\ZF$.
\begin{definition}
Let $\dot x$ be a $\QQ_\alpha$-name, we say that $\dot x$ is bounded by $f\in\prod\SC(\omega_\alpha)$, if whenever $\gaut{\vec\pi}\varrho_\alpha\restriction(A,E)\forces_{\QQ_\alpha}\dot y\in\dot x$, then $\gaut{\vec\pi}\varrho_\alpha\restriction(E\cap f\da,A)\forces_{\QQ_\alpha}\dot y\in\dot x$; similarly, $\dot x$ is bounded by $\delta<\alpha$ if $\gaut{\vec\pi}\varrho_\alpha\restriction(E,A\cap\delta)\forces_{\QQ_\alpha}\dot y\in\dot x$.
\end{definition}
\begin{lemma}\label{lemma:limit-sym-dist}
Suppose that $\dot x$ is a $\QQ_\alpha$-name for a subset of the ground model (in this case, $\IS_\alpha$) such that every $\dot y$ appearing in $\dot x$ is a canonical ground model name (i.e., a name from $\IS_\alpha$).
\begin{enumerate}
\item If $\dot x\in\HS$, then $\dot x$ is bounded by some $f\in\prod\SC(\omega_\alpha)$,
\item and if $\forces_{\QQ_\alpha}\rank(\dot x)<\check\omega+\check\alpha$, then there is some $\delta<\alpha$ such that $\dot x$ is bounded by $\delta$.
\end{enumerate}
\end{lemma}
The proof of \autoref{lemma:limit-sym-dist} is the same proof as \autoref{lemma:omega-sym-dist}. And it implies the following corollary in the same manner.
\begin{corollary}
If $\dot x\in\IS_{\alpha+1}$ and $p\forces_{\alpha+1}^\IS\rank(\dot x)<\check\omega+\check\alpha$, then there is some $q\leq_{\alpha+1} p$ and $\dot x'\in\IS_\alpha$ such that $q\forces_{\alpha+1}^\IS\dot x=\dot x'$.\qed
\end{corollary}
\begin{proposition}\label{prop:coherence}
For all $f\in\prod\SC(\omega_\alpha)$, $\dot\varrho_{\alpha,f}\in\IS_{\alpha+1}$.
\end{proposition}
\begin{proof}
It is not hard to see that $K_{0,f+1}$ witnesses that $\dot\varrho_{\alpha,f}$, as a $\QQ_\alpha$-name in $\IS_\alpha$ is in $\HS$. Therefore the conclusion follows.
\end{proof}
\subsection{General successor of limits}
Finally, we deal with the case of a successor of a limit ordinal. The idea is almost the same as the case $\omega+1$, or generally a successor ordinal. We remark that interestingly enough, there is absolutely no difference between the case where $\alpha$ is a limit of limits, a successor limit (i.e., $\delta+\omega$) or even an inaccessible cardinal. For readability, we still assume that $\alpha$ is a limit ordinal. We will again omit the $\alpha$ superscript from the elements of $F^\alpha$, writing $f_\eta$ instead of $f_\eta^\alpha$.

For every $\eta<\omega_{\alpha+1}$, we say that $\dot x\in\IS_{\alpha+1}$ is an almost $f^\eta$-name if there is some $f\in\prod\SC(\omega_\alpha)$ such that $\dot x$ is a $\PP_\alpha\ast\dot\QQ_{\alpha,f}$-name and $f=^*f_\eta$. We define $\dot R_\eta$ as before,
\[\dot R_\eta=\left\{\dot x\in\IS_{\alpha+1}\middd\begin{array}{l}
\dot x\text{ is an almost }f_\eta\text{-name or rank }\leq\omega+\alpha,\text{ and}\\
\text{every name appearing in }\dot x\text{ is from }\IS_\alpha.
\end{array}\right\}^\bullet.\]
\begin{proposition}
Suppose that $\gaut{\vec\pi}\in\cG_{\alpha+1}$, then $\gaut{\vec\pi}\dot R_\eta=\dot R_{\iota(\vec\pi^\alpha)(\eta)}$.\qed
\end{proposition}
In other words, the permutations in $\cG_\alpha$ have no effect on $\dot R_\eta$, and the effect of $\vec\pi^\alpha$ is by the permutation it implements (recall that $\vec\pi^\alpha$ is the permutation from $\sG_\alpha$).\smallskip

Next, we define $\dot\varrho_{\alpha+1}$ as $\tup{\dot R_\eta\mid\eta<\omega_{\alpha+1}}^\bullet$. The following proposition is proved exactly like \autoref{prop:omega-successor-conditions}.
\begin{proposition}
For every $A\in\jbd(\omega_{\alpha+1})$, $\dot\varrho_{\alpha+1}\restriction A\in\IS_{\alpha+1}$.\qed
\end{proposition}
Now we can define $\dot\QQ_{\alpha+1}$ as $\{\gaut{\vec\pi}\dot\varrho_{\alpha+1}\restriction A\mid\gaut{\vec\pi}\in\cG_{\alpha+1}, A\in\jbd(\omega_{\alpha+1})\}^\bullet$, ordered by reverse inclusion.
The definitions of $\sG_{\alpha+1}$ and $\sF_{\alpha+1}$ are as with the general successor case, defined by the permutable family on $\omega_{\alpha+1}$. We only need to prove the upwards homogeneity of $\PP_{\alpha+1}\ast\dot\QQ_{\alpha+1}$, and this is done in a similar manner to \autoref{prop:omega+1-up-hom}.
\subsection{Conclusions}
We have shown that $\tup{\varrho_\alpha\mid\alpha\in\Ord}$ is a symmetrically generic filter, and certainly it lies within $L[\varrho_0]$, our original Cohen extension. Each step is homogeneous, and no sets of rank $\omega+\eta$ are added after the $\eta+1$ step of the iteration. Therefore the preservation theorem tells us that indeed the resulting model is a model of $\ZF$, even without knowing it was bounded inside a Cohen real.
\section{Some general peculiar consequences}
\begin{definition}
We say that $A$ is an \textit{$\alpha$-set of ordinals} if there is some ordinal $\eta$ such that $A\subseteq\power^\alpha(\eta)$.
\end{definition}
We will abbreviate and write $\alpha$-sets to denote $\alpha$-sets of ordinals. The usual coding arguments used in $\ZFC$ for coding tuples of ordinals as ordinals extend to $\alpha$-sets, so if $A$ is an $\alpha$-set, there is a canonical way to represent $A^{<\omega}$ and $[A]^{<\omega}$ as $\alpha$-sets as well.
\begin{definition}
$\KWP_\alpha$ is the statement that every set is equipotent to an $\alpha$-set. We use $\KWP$ to denote $\exists\alpha\ \KWP_\alpha$.
\end{definition}
So $\KWP_0$ is the same as $\AC$, and $\KWP_1$ implies that every set can be linearly ordered. The proof of the next theorem can be found in \cite[Theorem~10.3]{Karagila:2016}
\begin{theorem}[The Generalized Balcar--Vop\v{e}nka--Monro Theorem]\label{thm:gen bal-vop-mon}
Suppose that $M$ and $N$ are two transitive models of $\ZF$ with the same $\alpha$-sets. If $M\models\KWP_\alpha$, then $M=N$.\qed
\end{theorem}
We will also mention the principle $\SVC$ formulated by Andreas Blass in \cite{Blass:1979}, which states that there is some set $X$ such that for any set $A$, there is some ordinal $\alpha$ and a surjection from $\alpha\times X$ onto $A$. We say that $X$ is a \textit{seed} and write $\SVC(X)$ to explicitly mention it, so $\SVC$ would simply be $\exists X\ \SVC(X)$. Blass showed that $\SVC$ is equivalent to the statement that there is a forcing extension in which $\AC$ holds.
\begin{theorem}[Theorem 10.4 in \cite{Karagila:2016}]
If $M\models\SVC$, then $M\models\KWP$.\qed
\end{theorem}
All symmetric models satisfy $\SVC$, as do $L(x)$ and $\HOD(x)$ for any set $x$.
\subsection{Kinna--Wagner principles in the Bristol model}
We denote by $M$ the Bristol model, and by $M_\alpha$ the $\alpha$th stage of the construction (i.e., the names in $\IS_\alpha$). When we write a forcing statement, or $\IS$ without an index, we mean of course to the class forcing of the entire construction.
\begin{theorem}\label{thm:alpha sets}
Suppose that $A\in M$ is an $\alpha$-set, then $A\in M_{\alpha+1}$.
\end{theorem}
\begin{proof}
We prove the statement by induction on $\alpha$. Suppose that $\dot A$ is a name such that $p\forces^\IS``\dot A$ is an $\check\alpha$-set of ordinals''. By the Symmetry Lemma, we can assume $\dot A$ by a name such that all the names which appear in it are from $\IS_\alpha$ (so for $\alpha=0$ these are all canonical names for ordinals). 

Let $\vec H$ be an excellent support for $\dot A$. We can assume, without loss of generality that $p$ has the following property (otherwise shrink $\vec H$ and extend $p$ once):
\begin{itemize}
\item For every $\xi$, if $\dom p_{\xi+1}$ is some $\eta_{\xi+1}$, then $H_\xi=\fix(\cB_\xi)$ for some disjoint approximation of $\{A^\xi_\gamma\mid\gamma<\eta_\xi\}$ and $H_{\xi+1}$ fixes $p_{\xi+1}$.
\item For every limit $\xi$, $\dom p_\xi$ is $(f\da,\delta_\xi)$ where $f$ is an upper bound to $\vec H\restriction\xi$ and $\delta_\xi=\max C(\vec H\restriction\xi)+1$, namely for all $\gamma<\delta_\xi$, $H_\gamma=\fix(\cB_\gamma)$ where $\cB_\gamma$ is a disjoint approximation for $\{A^\gamma_\beta\mid\beta<f(\beta)\}$.
\end{itemize}
This property implies that if $\vec\pi\in\vec H$, then $\gaut{\vec\pi}p=p$. Moreover, for every $\gamma$, we can make any two extensions of $p_\gamma$ compatible using an automorphism from $H_\gamma$.

Suppose now that $q\leq p$ and $q\forces\dot a\in\dot A$, with $\dot a\in\IS_\alpha$, and let $q'$ be any extension of $q\restriction\alpha+1$, then we will find some $\vec\pi\in\vec H$ such that $\gaut{\vec\pi}q$ is compatible with $q'$ and $\vec\pi\restriction\alpha=\id$. This will then imply that $\dot A'=\{\tup{q\restriction\alpha+1,\dot a}\mid\tup{q,\dot a}\in\dot A\}$ is such that $p\forces\dot A=\dot A'$ and that $\dot A'\in\IS_{\alpha+1}$ as wanted.

To find such $\vec\pi$, proceed by induction on $\supp q\setminus\alpha+1$, and in each step find an automorphism in $\vec H\restriction\gamma$ which only ``corrects'' the incompatible part of the $\gamma$th coordinate.
\end{proof}
\begin{corollary}\label{cor:KW failures}
$M$ does not satisfy $\KWP$, and for every $\beta$ there is some $\alpha$ such that $M_\alpha\not\models\KWP_\beta$.
\end{corollary}
\begin{proof}
Assume otherwise, then there is some $\alpha$ such that $M_\beta\models\KWP_\alpha$ for all $\beta$. But then for all $\beta>\alpha+1$, $M_\beta$ has the same $\alpha$-sets as $M_{\alpha+1}$, and by \autoref{thm:gen bal-vop-mon} this means $M_\beta=M_{\alpha+1}$. The same argument implies that $M$ cannot satisfy $\KWP$.
\end{proof}
\begin{corollary}
There is no forcing $\PP\in M$ which forces $\AC$ over $M$.\qed
\end{corollary}
\begin{corollary}
There is no $x\in M$ such that $M=L(x)$.
\end{corollary}
\begin{proof}
Given any $x$, then in $L(x)$ adding a bijection between the transitive closure of $x$ and $\omega$ will force the axiom of choice. In contradiction to the above corollary. Alternatively, note that $L(x)$ satisfies $\KWP_\alpha$ for some $\alpha>\rank(x)$.
\end{proof}
\section{In search for better grounds}
The original motivation for the construction was the relation between the HOD Conjecture and the Axiom of Choice Conjecture. We give a simplified version of the HOD Conjecture to avoid the technicalities.
\begin{definition}[Theorem~19 in \cite{WDR:2016}]
The \textit{HOD Conjecture} asserts, assuming the existence of an extendible cardinal $\delta$, that for every singular $\lambda>\delta$, $\lambda$ is singular in $\HOD$ and $\lambda^+=(\lambda^+)^{\HOD}$. 
\end{definition}
\begin{definition}[Definition~29 in \cite{WDR:2016}]
The \textit{Axiom of Choice Conjecture} asserts, that assuming $\ZF$ and that $\delta$ is an extendible cardinal, then the Axiom of Choice holds in $V[G]$, where $G$ is $V$-generic for collapsing $V_\delta$ to be countable.
\end{definition}

The Bristol workshop in 2011 was aimed at looking into the possibility that the Axiom of Choice Conjecture follows from the HOD Conjecture; and the Bristol model was a warm-up exercise into a possible way for disproving such statement. Of course, these conjectures rely on the existence of an extendible cardinal, which is certainly very far from anything we can get in $L$. And indeed extendible cardinals, and even the weaker supercompact cardinals, already imply that singular cardinals do not have weak squares or even permutable scales. So the construction of the Bristol model fails.

\subsection{Better grounds for Bristol models}
The construction as described by the Bristol group assumed $V=L$ for the ground model. Reviewing the details of the construction given in this paper, we can see that only three assumptions were actually used in the construction: $\GCH$, $\square^*_\lambda$ for singular $\lambda$ (or rather, the existence of permutable scales), and global choice. These three assumptions hold in much richer models than just $L$. They hold in any of the $\GCH$-preserving forcing extensions of $L$; and in canonical inner models of large cardinals like $L[U]$, or other extender based models.

Let $T$ denote $\ZFC+\GCH+\forall\lambda(\cf(\lambda)<\lambda\rightarrow\square^*_\lambda)$.\footnote{We can omit global choice from the assumption by the following argument: pass to the model of von Neumann--G\"odel--Bernays set theory by taking definable classes, add a generic global choice function without adding sets, repeat the construction, and now forget about the additional classes.}\textsuperscript{,}\footnote{The role of $\GCH$ is slightly trickier. It seems that it can be omitted in favor of allowing ``gaps'' in the cardinals used for each step of the construction. This, however, might cause sets to be added to low-ranked levels in an uncontrolled way. However, its use in the proof that permutable scales exist for inaccessible cardinals can also be reduced to $\mathfrak b_\lambda=\lambda^+$ which is known to be consistent with the failure of $\CH$ at $\lambda$.} Then what we actually showed in this paper so far is that if $V\models T$, and $c$ is a Cohen generic real over $V$, then there is a Bristol model intermediate to $V$ and $V[c]$. Of course, the real challenge is to preserve information from $V$ to Bristol models over $V$.

The following is an immediate corollary from the construction and \autoref{thm:alpha sets}.
\begin{theorem}
Let $\kappa$ be a large cardinal whose property is defined by $\alpha$-sets for a fixed $\alpha<\kappa$ and is preserved under forcings of rank $<\kappa$. Then $\kappa$ preserves the property when moving to the Bristol model.\qed
\end{theorem}
\begin{corollary}
If $\kappa$ is strong (characterized by the existence of extenders), measurable, Ramsey, weakly compact, Mahlo, inaccessible, then it also has the property in the Bristol model.\qed
\end{corollary}
One could argue that the correct definition of measurable, or even weakly compact, is using elementary embedding. We are optimistic about the conjecture that these embeddings extend to the Bristol model, if they exist in the ground model, in a way ``visible'' to the Bristol model. This should be provable by extending the work of the author with Yair Hayut on lifting elementary embeddings to symmetric extensions from \cite{Hayut-Karagila:2018}.\medskip

This construction, however, is not always doable. The assumptions that there exists a permutable scale is inevitably going to fail, as it implies the principle called $\mathrm{ADS}_\kappa$, which in turn fails above a supercompact (this is discussed in \cite{CFM:2001}). So even though the Bristol model seems to be quite capable of accommodating large cardinals, the construction, as given here, cannot go past a certain point when assuming very large cardinals, or even $\aleph_\omega$ under the assumption of Martin's Maximum (even ignoring the $\GCH$ requirement).

\section{Open problems}
We finish this paper by discussing some problems of interest for future work.
\subsection{Fragments of choice}
The first and foremost question is what fragments of choice still hold in the Bristol model. Does the Boolean Prime Ideal theorem hold? Can every set be linearly ordered? Does the axiom of choice for pairs hold? Does $\DC$ hold in the Bristol model, and can we modify the construction to obtain an arbitrarily $\DC_\kappa$?
\subsection{The genericity problem}
Assuming $\GCH$, the Boolean completion of the Cohen forcing has size $\aleph_1$, and therefore has at most $\aleph_2$ automorphisms, and therefore at most $\aleph_3$ normal filters of subgroups. This means that by $M_{\omega_4}$, most of the models cannot be obtained as symmetric extensions of the Cohen forcing itself.

This has an interesting consequence that by adding one Cohen real to $L$, most of the intermediate models we have added are certainly not symmetric extensions. And indeed, from a Platonic point of view, it means that if one takes the axiom $V=L[c]$ to be true, then we have in fact proved that an iteration of symmetric extensions over inner models need not be a symmetric extension of an inner model. However, if we look at set models, or at a far more ``flexible'' philosophical approach, we have to wonder, are these intermediate models symmetric extensions of the Levy collapse?
\subsection{The choice of permutable families and structure of Bristol models}
We began the construction by fixing a permutable family for every cardinal. It is easy to see that by taking an equivalent permutable family, we will invariably get the same model. But what if we take smaller permutable families, or larger permutable families? Is there a nice structure theorem, and can we perhaps code things into the maximality of the families we used in each step?

In the same breath, it is obvious that every real in the Bristol model is a Cohen real over the ground model. So we can re-interpret the Bristol model from that real. This, in conjunction with the above, leads to the question, what is the order of the Bristol models? Is there a proper-class of pairwise distinct Bristol models?
\subsection{Other type of Bristol models?}
The construction was made possible in part due to the fact that Cohen reals have a particular wealth when it comes to inner models, even those satisfying $\AC$. Can we begin the construction using a random real? If so, will the successive steps be Cohen-like, similar to this construction, or can we find random-Bristol models made of something akin to a generalized random forcing?
\subsection{What happens with very large cardinals?}
As we remarked, the historical motivation of the Bristol workshop was to see if the Axiom of Choice Conjecture follows from the HOD Conjecture. The Bristol model itself was an exercise to try and show that the Axiom of Choice Conjecture does not follow from the HOD Conjecture. The construction itself falters at the level of supercompact cardinals, as they imply the failure of square principles. So reaching an extendible cardinal is still quite a long way to go.

However, the argument that we conjecture can used to conclude that an ultrapower elementary embedding of a measurable cardinal extends to an amenable embedding of the Bristol model to the Bristol model of the target model, seems to be strong enough to preserve even elementary embeddings witnessing extendibility.

One can only ask whether or not it is possible that the failure of squares---or rather the nonexistence of permutable scales---implies some type of compactness which can be used to carry the construction through these problematic levels. If that is the case, then one has a refutation of the Axiom of Choice Conjecture. Adding this to the very strong evidence that the HOD Conjecture might imply the Axiom of Choice Conjecture, and we get a refutation of the HOD Conjecture.

The natural question, if so, is what sort of compactness one gets from the nonexistence of permutable scales for singular cardinals, and can that be used to handle the limit steps of the Bristol construction?
\section*{Acknowledgments}
Having spent nearly four years developing the needed tools and fine-tuning the construction, there are many to whom the author owes mention. First and foremost to the Bristol group, all of which shared their experiences, understandings, and helped whenever possible. Specific thanks go to Menachem Magidor, the author's advisor for his constant advice on this work, and for helping to see it through. To Yair Hayut for his interest in the project, and his many helpful suggestions. In addition to that, to the organizers of the HIF programme for their invitation to Cambridge in October 2015, and to Peter Koepke for his generous invitation to Bonn during April 2016 where the fresh spring air helped shape some of the most crucial lemmas and their proofs. The author would also like to thank the referee for their helpful remarks that made this paper better. A final acknowledgement goes to Matti Rubin who did not live to see this project through, for listening several times to the basic ideas, the outlined details, and his careful advice.
\bibliographystyle{amsplain}
\providecommand{\bysame}{\leavevmode\hbox to3em{\hrulefill}\thinspace}
\providecommand{\MR}{\relax\ifhmode\unskip\space\fi MR }
\providecommand{\MRhref}[2]{%
  \href{http://www.ams.org/mathscinet-getitem?mr=#1}{#2}
}
\providecommand{\href}[2]{#2}

\end{document}